\documentclass[reqno]{amsart}

\usepackage{graphicx,pdfsync,amssymb, latexsym,amsfonts,amsbsy,color,subcaption,esint,mathtools,nicefrac,tikz}
\usepackage{hyperref}
\hypersetup{
colorlinks=true,
linkcolor=blue,
filecolor=blue,      
urlcolor=blue,
citecolor=blue,
}
\usepackage{enumitem}
\setitemize{itemsep=0.5em}
\setenumerate{itemsep=0.5em}

\usepackage[utf8]{inputenc}% for special characters
\usepackage[T1]{fontenc}% for special characters

\usepackage[left=1.5in, right=1.5in]{geometry}

 \DeclareMathOperator{\curl}{curl}
\DeclareMathOperator{\Supp }{supp}
\DeclareMathOperator{\Id }{Id} 
 
\DeclareMathOperator{\D}{div}
\DeclareMathOperator{\Tr}{Tr} 
\DeclareMathOperator{\dist}{dist}

\DeclareMathOperator{\Lin}{lin} 
\DeclareMathOperator{\Osc}{osc} 
 
\DeclareMathOperator{\Rem}{rem}

\DeclareMathOperator{\Cor}{cor} 
\DeclareMathOperator{\Lap}{lap} 
\DeclareMathOperator{\Acc}{acc} 
\DeclareMathOperator{\Dri}{dri}

\newtheorem{theorem}{Theorem}[section]
\newtheorem{lemma}[theorem]{Lemma}
\newtheorem{proposition}[theorem]{Proposition}
\newtheorem{definition}[theorem]{Definition}
\newtheorem{corollary}[theorem]{Corollary}
\newtheorem{remark}[theorem]{Remark}

\usepackage{times}
\usepackage{setspace}

\def \TT  {\mathbb{T}} %torus
\def \RR {\mathbb{R}}  %real numbers
\def \NN {\mathbb{N}}  %natural numbers
\def \ZZ {\mathbb{Z}}  %integer numbers

\def \p {\partial}

\def \ep {\varepsilon}

\def \k {\kappa}
\def \l {\lambda}
\def \L {\Lambda}

\def \ek {\mathbf{e}_{k}}

\def \bwk {\mathbf{ W}_{k}}
\def \bpk {\mathbf{\Psi}_{k}}

\numberwithin{equation}{section}

\begin{document}

\title[Nonuniqueness for 2D NSE]{$L^2$-critical nonuniqueness for the 2D Navier-Stokes equations}

\author{Alexey Cheskidov}
\address{Department of Mathematics, Statistics and Computer Science,
University of Illinois At Chicago, Chicago, Illinois 60607}
\email{acheskid@uic.edu}

\author{Xiaoyutao Luo}

\address{Department of Mathematics, Duke University, Durham, NC 27708}

\email{xiaoyutao.luo@duke.edu}

 %  General info

\date{\today}

\begin{abstract}
In this paper, we consider the 2D incompressible Navier-Stokes equations on the torus. It is well known that for any $L^2$ divergence-free initial data, there exists a global smooth solution that is unique in the class of $C_t L^2$ weak solutions.  We show that such uniqueness would fail in the class $C_t L^p$ if $  p<2$. The non-unique solutions we constructed are almost $L^2$-critical in the sense that $(i)$ they are uniformly continuous in $L^p$ for every $p<2$; $(ii)$ the kinetic energy agrees with any given smooth positive profile except on a set of arbitrarily small measure in time.
 
\end{abstract}
\maketitle

%%%%%%%%%%%%%%%%%%%%%%%%%%%%%%%%%%%%%%%%%%%%%%%%%%%%%%%%%%%%%%%%%%%%%%%%%%%%%%%%%%%%%%%%%%%%%%%%%%%%%%%%%%%%%%%%%%%%%%%%%%%%%%
\section{Introduction}\label{sec:intro}
%%%%%%%%%%%%%%%%%%%%%%%%%%%%%%%%%%%%%%%%%%%%%%%%%%%%%%%%%%%%%%%%%%%%%%%%%%%%%%%%%%%%%%%%%%%%%%%%%%%%%%%%%%%%%%%%%%%%%%%%%%%%%%

\subsection{Statement of the problem}
The 2D Navier-Stokes equations are a fundamental mathematical model of fluid flow written as
\begin{equation}\label{eq:NSE}
\begin{cases}
\p_t u -  \Delta u + \D( u \otimes u) + \nabla p = 0 &\\
\D u = 0&\\
u(0) = u_0
\end{cases}
\end{equation}
posed on a spatial domain $\Omega \subset \RR^2$ with a suitable boundary condition, where $u :\Omega \times [0,T]  \to \RR^2 $ is the unknown velocity with initial data $u_0$ and $p : \Omega \times [0,T]  \to \RR  $ is a scalar pressure. We consider the Cauchy problem of \eqref{eq:NSE} on a time interval $[0,T]$ for some initial data $u_0$ and $T>0$. The existence and uniqueness of smooth solutions have been proved by Leray \cite{Leray1933} for $\Omega =\RR^2$, and Ladyzhenskaya \cite{MR0108963} for bounded domains.

In the paper, we focus on the periodic case $\Omega = \TT^2  =\RR^2 /\ZZ^2 $ and the solutions with zero spatial mean
$$
\int_{\TT^2} u(x,t) \, dx = 0,
$$
which is also conserved under the evolution of the equations \eqref{eq:NSE}.

For any divergence-free initial data $u_0 \in L^2(\TT^2)$, a standard Galerkin method leads to  the existence of a weak solution $u \in L^\infty([0,T];  L^2(\TT^2))   \cap L^2([0,T];  H^1(\TT^2))$, and the validity of the energy equality
\begin{equation}\label{eq:energy_equality}
	\frac{1}{2} \|u(t) \|_2^2 + \int_0^t\|\nabla u (s)\|_2^2 \,ds = \frac{1}{2} \|u(0) \|_2^2 \quad \text{for all $t \geq 0$},
\end{equation}
follows from the Onsager criticality of the energy-enstrophy space in 2D. A weak-strong uniqueness argument and  Ladyzhenskaya's inequality imply the regularity and uniqueness of such Leray-Hopf solutions.

In fact, the uniqueness of the 2D Navier-Stokes equation can be stated in a much stronger way. More precisely, Leray-Hopf solutions are unique in a much larger class than the Leray-Hopf class itself. As discussed in \cite{MR1891170},  the result \cite{MR1813331} of Furioli, Lemari\'{e}-Rieusset, and Terraneo implies that in 2D any weak solution in $C_t L^2$ is unique without any additional regularity assumptions:

\begin{theorem}[\cite{MR1813331}]\label{thm:CL2_uniqueness}
For any divergence-free $u_0 \in L^2(\TT^2)$, there is only one weak solution in the class $C_t L^2$ with initial data $u_0$.
\end{theorem}

Concerning the sharpness of this existence/uniqueness result, we ask two questions:
\begin{enumerate}[label=(\alph*)]
    \item Does the existence still hold if initial data $u_0 \in L^p(\TT^2)$ for $p<2$?
    
    \item Are weak solutions in the class $C_t L^p$ unique if $p<2$?
\end{enumerate}
 
The aim of this paper is to provide a positive answer to the first question, and a negative answer to the second one. In fact, the main results in this paper will be stated in wider Sobolev spaces, but we choose the usual Lebesgue spaces $L^p$ here as the most elementary and significant. For instance, the space  $  L^2 (\TT^2)$ is invariant under the natural scaling of the Navier-Stokes equations   $ u( x,  t ) \to u_\l( \l x, \l^2 t)$, and the square of $L^2(\TT^2)$ norm represents the total kinetic energy, which is nonincreasing in time for smooth solutions due to \eqref{eq:energy_equality}.

It is easy to see that when $p<2$, the regularity $ u\in C_t L^p$ itself alone is not enough to make sense of the weak formulation. By a weak solution of \eqref{eq:NSE}, we mean a vector field $ u \in L^2 (  \TT^2\times [0,T])$ satisfying all of the following conditions:
\begin{enumerate}[label=(\arabic*)]

\item For $a.e \, t\in [0,T]$, $ u(\cdot, t)$ is weakly divergence-free;

\item For any  $\varphi \in \mathcal{D}_T $,
\begin{equation}
  \int_{\TT^2} u_0(x) \varphi( \cdot  , t  ) \, dx + \int_0^T \int_{\TT^2} u\cdot (\Delta \varphi + u\otimes u : \nabla \varphi - u \cdot \partial_t \varphi ) \, dx dt =0,
\end{equation}
 
\end{enumerate}
where $\mathcal{D}_T \subset  C^\infty (  \TT^2\times \RR )$ is the space of all divergence-free  test functions vanishing on $ t\geq T$.
 
The above definition seems too weak at first glance. However, by \cite[Theorem 2.1]{MR316915}, the above definition of weak solutions is equivalent to the integral equation
\begin{equation}\label{eq:mild_formulation}
u(t) = e^{t\Delta}u_0 - \int_0^t e^{(t-s)\Delta} \mathcal{P}\D(u\otimes u) (s)\, ds ,
\end{equation}
where $ e^{t\Delta}$ is the heat semigroup and $\mathcal{P}$ is the Leray projection onto the divergence-free vector fields. This equation \eqref{eq:mild_formulation} is often referred to as the \emph{mild formulation} of the Navier-Stokes equations and has been used to construct unique local solutions, also known as the \emph{mild solutions}, when the initial data is critical or subcritical. We will discuss this in the summary below.

The following theorem is one of the main results of this paper, which shows both the existence and nonuniqueness of weak solutions in $C_t L^p $ for $p< 2$.

\begin{theorem} \label{thm:Energy_Profile_short}
Fix $1  \leq p <2$ and let $u_0 \in L^p(\TT^2)$ be divergence free and $e :[0,T] \to \RR^+$ be strictly positive and smooth.

Then for any $\ep>0$, there exist a weak solution $u$ with initial data $u(0) =u_0$ and regularity 
\begin{equation*}
u \in C([0,T]; L^{p }(\TT^2)) \quad \text{and}\quad u \in C((0,T]; L^{p'}(\TT^2))  \quad\text{for all $p'<2$},
\end{equation*}
and a Borel set $\mathcal{G} \subset [0,T]$ with $\mathcal{L}^1([0,T] \setminus \mathcal{G})  < \ep$   such that
\[
\|u(t)\|_{L^2(\TT^2)}^2  = e(t), \qquad \forall t \in \mathcal{G}.
\]
\end{theorem} 

Theorem \ref{thm:Energy_Profile_short} contributes to both the existence and nonuniqueness in the  class $ C_t L^p $ for $1 \leq p <2$.  Thanks to the Lusin-type property of the kinetic energy of the constructed solutions, Theorem \ref{thm:Energy_Profile_short} is sharp in both of the following two aspects.
\begin{itemize}
    \item The regularity $C_t L^p$ for any $1 \leq p <2$  is sharp: if $p=2$, then any $C_t L^2$ weak solution is unique in the class $C_t L^2$ with the same initial data;
    \item The size of set $\mathcal{G}$ is optimal, namely,  $[0,T]\setminus \mathcal{G}$ can not be of measure zero: if $\mathcal{L}^1([0,T] \setminus \mathcal{G}) = 0$ and $\|e\|_{L^\infty}$ were small, then $u$ would have a small $L^\infty_t L^2$ norm, which implies uniqueness~\cite{MR1891170}. 
\end{itemize}

%%%%%%%%%%%%%%%%%%%%%%%%%%%%%%%%%%%%%%%%%%
\subsection{Summary of existence and uniqueness results}
%%%%%%%%%%%%%%%%%%%%%%%%%%%%%%%%%%%%%%%%%%
Before we introduce the main theorems, let us review the existing existence and uniqueness results for the 2D Navier-Stokes equations.

Since the classical work of Leray and Ladyzhenskaya, there has been a vast amount of literature on extending the existence and uniqueness results to a wider class of solutions.

\subsubsection*{Existence}
Since the global existence of Leray-Hopf solutions \cite{MR1555394}, the existence theory for 2D Navier-Stokes equations mainly focused on two aspects: using the mild formulation for initial data in scale-invariant spaces or using the vorticity formulation in the case where the initial vorticity is a measure.

Using the mild formulation \eqref{eq:mild_formulation} for the velocity, in the celebrated work \cite{MR1808843}  Koch and Tataru showed the existence of a unique solution for small initial data $u_0 \in \text{BMO}^{-1}$ in dimensions $d \geq 2$. In 2D, Gallagher and Planchon \cite{MR1891170} constructed global unique solutions in scale invariant Besov spaces $u_0 \in B^{-1+2/p}_{p,q}$ for any finite $p,q$. This was later extended by Germain \cite{MR2135239} to $u_0 \in \text{VMO}^{-1} $, where the space $\text{VMO}^{-1} $ is the closure of Schwartz functions in $\text{BMO}^{-1} $. We also mention related works \cite{MR968416,MR1860126,MR1946395} and refer to ~\cite{MR2768550,MR3469428} for the background in Fourier analysis in line of research.

At the level of vorticity $\omega =\curl u$, the critical scaling corresponds to $\omega_0 \in L^1$, which suggests the consideration of measure-valued vorticity. The first existence result in this direction dates back to \cite{MR794825,MR1017289} where the existence of weak solutions was shown when $\omega_0 $ is a Radon measure. We refer to \cite{MR1308857,MR853597,MR1270113} and reference therein for this line of research. Heuristically, such results would correspond to the velocity of scaling $u_0 \in B^{-1+ 2/p}_{p,q}$
for $p\leq 2$ and $q=\infty$.

\subsubsection*{Uniqueness}

For the velocity formulation, the known uniqueness results, which are usually corollaries of the existence results, can be summarized as follows: if $u  \in C_t B^{-1+2/p}_{p,q}$ with finite $p,q$, then $u$ is unique in the class of weak solution with regularity $  C_t B^{-1+2/p}_{p,q}$. It should be noted that even though the existence results in \cite{MR1808843,MR2135239} surpass \cite{MR1891170}, the uniqueness in \cite{MR1808843,MR2135239} can not be stated with only one continuity condition in time.

The uniqueness of the vorticity formulation however does not follow from the construction of the solutions. Indeed, in \cite{MR794825,MR1017289,MR1308857,MR853597,MR1270113}, the uniqueness results all have either smallness or structural assumptions. Building upon \cite{MR2176270}, this problem has finally been settled in \cite{MR2178064} : if $\omega_0 \in \mathcal{M}$, then there is only one solution $\omega(t)$ which is continuous in $L^1\cap L^\infty$ for $t>0$.
 
The existence and uniqueness of weak solutions for the 2D Navier-Stokes equations have more or less matured except at the endpoint $B^{-1 }_{\infty,\infty}  $, which is the largest critical space with respect to the scaling the Navier-Stokes equations, independent of the space dimensions. The current existence results in this space require extra assumptions (see for instance \cite{MR2290141,MR2776367,MR3177282} and reference therein). It is also not clear at the moment whether weak solutions in $C_t B^{-1 }_{\infty,\infty}$ are unique, though other types of ill-posedness results are available in dimensions $d \geq 3$ \cite{MR2473255,MR2473256,MR2566571}.

Supported by the development of the existence and uniqueness theories so far, we can summarize the existence and uniqueness in a non-rigorous principle:
\begin{center}
    \textit{Existence and uniqueness may hold in the critical and sub-critical regime.}
\end{center}
 
This heuristic principle has also been proven quite successful in view of the recent development of the negative results thanks to the convex integration technique.

%%%%%%%%%%%%%%%%%%%%%%%%%%%%%%%%%%%%%%%%%%
\subsection{Convex integration technique}
%%%%%%%%%%%%%%%%%%%%%%%%%%%%%%%%%%%%%%%%%%

Convex integration is a technique originated in isometric embedding problems of geometry dating back to the work of Nash \cite{MR0065993} and Kuiper \cite{MR0075640}. Its application to fluid dynamics was brought to attention only in 2009 by the pioneering work of De Lellis and Székelyhidi Jr. \cite{MR2600877}. Since its inception in \cite{MR2600877}, this technique has been proven very fruitful in the fluids community. Remarkably, its development over a series of works \cite{MR2600877,MR3090182,MR3254331,MR3374958,MR3530360,1701.08678} has culminated in the resolution of the Onsager conjecture for the 3D Euler equations by Isett
\cite{MR3866888}. Its application to other models is a very active research direction, and we refer readers to \cite{MR2813340,MR3479065,MR3987721,MR4138227,MR4126319,MR4105741,MR4198715,2101.09278} and surveys \cite{MR3929468,MR4188806} reference therein for other interesting work on convex integration.

Very recently, in a groundbreaking work \cite{MR3898708}, Buckmaster and Vicol successfully applied an ``intermittent'' variant of this technique to construct non-unique weak solutions of the 3D Navier-Stokes equations. Following this breakthrough development, there have been multiple results in viscous settings, such as extensions of \cite{MR3898708} to the hyperdissipative case~\cite{1808.07595,MR4097236}, non-unique weak solutions with partial regularity in time~\cite{1809.00600}, nonconservative $H^{1/2-}$ solutions of the Euler equations \cite{2101.09278}, existence of nontrivial stationary weak solutions of $d$-dimensional Navier-Stokes equations for $d\geq 4$ \cite{MR3951691}, nonuniqueness of Leray-Hopf solutions for hypodissipative Navier-Stokes equations~\cite{MR3843425,MR3941228}, Hall-MHD~\cite{1812.11311}, and some power-law flows~\cite{arXiv:2007.08011} in various solution classes where the Leray structure theorem does not hold.

Ultimately, any attempt of using the  strategy of \cite{MR3898708}  to extend the nonuniqueness  to the 2D Navier-Stokes equations would face substantial difficulty. In a nutshell, this is because the spatially intermittent framework of \cite{MR3898708} requires the underlined PDE to be at least $C_t L^2$-supercritical, which, in 2D, is a contradiction to Theorem \ref{thm:CL2_uniqueness}. In \cite{2004.09538,2009.06596}, the authors of the current paper developed a space-time intermittent variant of the framework of \cite{MR3898708} which allows for more flexibility in the scaling of the nonuniqueness range. Taking advantage of the added temporal intermittency, in \cite{2009.06596} we constructed nonunique weak solutions  to the Navier-Stokes equations in dimensions $d\geq 2$ in $L^p_t L^\infty$ for any $p<2$, proving the first nonuniqueness result for the 2D Navier-Stokes equations and the sharpness of the classical $L^2_t L^\infty$ Ladyzhenskaya-Prodi-Serrin criterion.

\begin{figure}[ht]
\centering 
\begin{subfigure}[b]{0.4\linewidth}
\begin{tikzpicture}[scale=0.75, every node/.style={transform shape}]

% horizontal axis
\draw[->] (0,0) -- (5.5,0) node[anchor=north] {\scriptsize Time scaling $\frac{1}{q}$};
% vertical axis
\draw[->] (0,0) -- (0,5) node[anchor=south ] {\scriptsize Space scaling $\frac{1}{p}$};

% hori_labels
\filldraw	(0,0) circle (1pt) node[anchor=north] {\scriptsize  $L^\infty_{t,x}$}
        (4,2) circle (1pt) node[anchor=north] {\scriptsize $L^2_t H^{1} \subset L^2_t L^{\frac{2d}{d-2}}$};
\draw (4,0) circle (1pt) node[anchor=north] {\scriptsize $L^2_tL^\infty$}
		(0,2.5) circle (1pt) node[anchor=east] {\scriptsize $C_tL^d$};

% verti_labels
\filldraw	
        (0,3.9) circle (1pt)
        (0,3.7) circle (1pt) 
        (0,4) circle (1pt) node[anchor=east] {\scriptsize $L^\infty_t L^2$};

% threshold label
%\draw[->,orange] (1.2,0.7) node[anchor=north]{{\tiny Uniqueness Threshold: $\frac{2}{p} + \frac{d}{q} = 1$}} .. controls +(up:0.5cm)  .. (1.5,1.5);
\draw[->,orange] (1.2,0.9) node[anchor=north]{
\renewcommand{\arraystretch}{0.55}
{\begin{tabular}{|c|} 
 \hline
 {\tiny Uniqueness threshold:}   \\ 
 {\tiny$\frac{2}{p} + \frac{2}{q} = 1$}\\
 \hline
\end{tabular}}
} .. controls +(up:0.5cm)  .. (1.5,1.5);

% LH label
\draw[->,orange] (3.5,3.2) node[anchor=south]{{\tiny Leray-Hopf scaling}} -- (2.7,2.7);

% nonuniqueness results
\draw[->]	(1,5) node[anchor=west] {\tiny \protect\cite{MR3898708}, $d=3$} -- (0.05,4.06);
\draw[->]   (1.2,4.5) node[anchor=west] {\tiny \protect\cite{1809.00600}, $d=3$} -- (0.05,3.93);
\draw[->]    (1.2,3.8) node[anchor=west] {\tiny \protect\cite{MR3951691}, $d=4$} -- (0.05,3.7);
 \draw[->]    (4.5,0.5) node[anchor=west] {\tiny \protect\cite{2009.06596}, $d \geq 2$} -- (4.05,0.05);
 
\draw[->,red]    (0.5,2.5) node[anchor=west] {\tiny [Theorem \ref{thm:Energy_Profile_short}, $d=2$]} -- (0.1,2.5);

% criticality_line
\draw[thick,blue] (0,2.5) -- (4,0);

% LH_line
\draw[thick,dashed,blue] (0,4) -- (4,2);
\end{tikzpicture}
\caption{$d \geq 2$. The Leray-Hopf line might be higher than $C_t L^d$.}
\end{subfigure}
\hspace{0.5cm}
\begin{subfigure}[b]{0.4\linewidth}
\begin{tikzpicture}[scale=0.75, every node/.style={transform shape}]

% horizontal axis
\draw[->] (0,0) -- (5.5,0) node[anchor=north] {\scriptsize Time scaling $\frac{1}{q}$};
% vertical axis
\draw[->] (0,0) -- (0,5) node[anchor=south] {\scriptsize Space scaling $\frac{1}{p}$};

% hori_labels
\filldraw	(0,0) circle (1pt) node[anchor=north] {\scriptsize  $L^\infty_{t,x}$};
\draw (4,0) circle (1pt) node[anchor=north] {\scriptsize $L^2_tL^\infty$};

% verti_labels
\draw	
		(0,2.5) circle (1pt) node[anchor=east] {\scriptsize $C_t L^2$};

% threshold label
%\draw[->,orange] (1.2,0.7) node[anchor=north]{{\tiny Threshold:\\ $\frac{2}{p} + \frac{2}{q} = 1$}} .. controls +(up:0.5cm)  .. (1.5,1.5);

\draw[->,orange] (1.2,0.9) node[anchor=north]{
\renewcommand{\arraystretch}{0.55}
{\begin{tabular}{|c|} 
 \hline
 {\tiny Uniqueness threshold:}   \\ 
 {\tiny$\frac{2}{p} + \frac{2}{q} = 1$}\\
 \hline
\end{tabular}}
} .. controls +(up:0.5cm)  .. (1.5,1.5);
 
% nonuniqueness results
 
 \draw[->]    (4.5,0.5) node[anchor=west] {\tiny \protect\cite{2009.06596}} -- (4.05,0.05);
 
\draw[->,red]    (0.5,3) node[anchor=west] {\tiny [Theorem \ref{thm:Energy_Profile_short}]} -- (0.05,2.55);

% criticality_line % LH label
\draw[->,orange] (3.5,1.5) node[anchor=south]{{\tiny Leray-Hopf scaling}} -- (2.6,0.95);

\draw[thick,blue] (0,2.5) -- (4,0);

\end{tikzpicture}
    \caption{$d = 2$. The Leray-Hopf scaling coincides with the critical scaling.}
\end{subfigure}
\caption{Uniqueness/nonuniqueness in $L^p_t L^q$ scale}
\label{fig:my_label}
\end{figure}
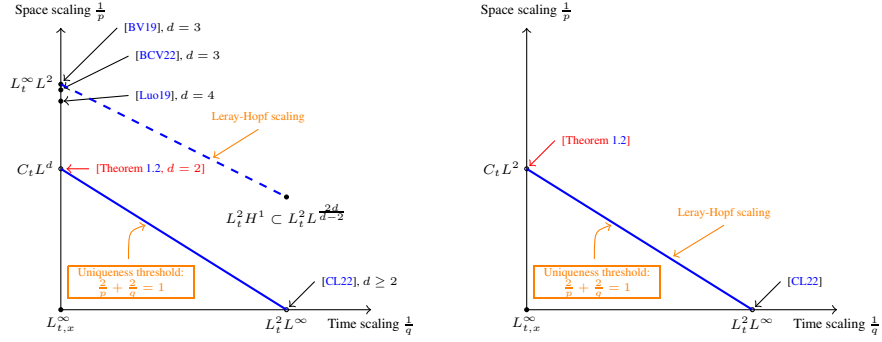 

While the results in \cite{2009.06596} demonstrate the sharpness of the Ladyzhenskaya-Prodi-Serrin criteria $L^p_tL^q$, $\frac{2}{p}+\frac{d}{q}=1$, at the $q=\infty$ endpoint, it remains open whether the rest of the cases in these uniqueness criteria are sharp, as shown in Figure \ref{fig:my_label}. In the current paper, we are able to reach the other endpoint $q=d$ in the case of two dimensions $d=2$.

%%%%%%%%%%%%%%%%%%%%%%%%%%%%%%%%%%%%%%%%%%
\subsection{Main results}
%%%%%%%%%%%%%%%%%%%%%%%%%%%%%%%%%%%%%%%%%%
In this subsection, we list additional results, complementary to Theorem~\ref{thm:Energy_Profile_short}, that follow directly from the construction.

\subsubsection{Density of energy near $t=0$}
First, in addition to the Lusin-type property for the kinetic energy, we have the density property \eqref{eq:density_L_2} at $t=0$, which we state in the case of $L^2$ initial data for simplicity.

\begin{theorem} \label{thm:Main-L^2}
Let $u_0 \in L^2(\TT^2)$  be divergence free and $e :[0,T] \to \RR^+$ be strictly positive, smooth and such that $e(0) =\|u_0\|_2$.

There exist a weak solution with initial data $u(0) =u_0$ and regularity
\begin{equation*}
u \in
C([0,T]; L^{p'}(\TT^2)) \quad\text{for all $p'<2$},
\end{equation*}
and a Borel set $\mathcal{G} \subset [0,T]$ such that
\[
\|u(t)\|_{L^2}^2 = e(t), \qquad \forall t \in \mathcal{G},
\]
and the density of $\mathcal{G}$ at $t=0$ is $1$:
\begin{equation}\label{eq:density_L_2}
\lim_{t \to 0+} t^{-1}\mathcal{L}^1(\mathcal{G}\cap [0,t]) =1.
\end{equation}

\end{theorem}

\subsubsection{Extension to rough initial data}
Our method of constructing weak solutions allows for initial data $u_0 $ to be much rougher than $L^1(\TT^2)$. Such a supercritical regularity persists over the whole time interval and is supplemented with an additional smoothing effect, a gain of almost two derivatives.

\begin{theorem} \label{thm:KT_sharpness}
For any divergence free $u_0 \in W^{s,1}(\TT^2)$ for some $s>-1$ there are infinitely many solutions in the class
$$
u \in C((0,T]; W^{s',1}(\TT^2))  \quad\text{for all $s'<1$}.
$$
\end{theorem}

\begin{remark} \label{remark:density_roughtID}
A few remarks are in order.
\begin{enumerate}
\item The density estimate \eqref{eq:density_L_2} is a much stronger result than the epsilon result in Theorem \ref{thm:Energy_Profile_short}: the kinetic energy of the solution will look more and more like $e(t)$ as $t \to 0+$.

\item For every constructed solution,
there is an increasing sequence of sets $\mathcal{T}_j \subset [0,1]$ with $\mathcal{L}^1(( \cup_j \mathcal{T}_j)^c  )=0$, such that $u$ continuously maps $\mathcal{T}_j \to L^2(\TT^2)$ for every $j \in \NN$. Here the sets $\mathcal{T}_j$ are equipped with the induced Euclidean metric.

\item  Note that $u \in C((0,T]; W^{s,1}(\TT^2))$ for all $s<1$. In addition, since the building blocks are fully intermittent in space, a similar argument implies that the vorticity $\omega=\curl u\in C((0;T];L^{p})$ for all $p<1$, which means that it is ``almost'' in $L^1 $,  consistent with the uniqueness at the level of the vorticity formulation.

\item All solutions constructed in this paper are nowhere regular in time, namely, there does not exist an interval $I\subset [0,T]$ such that $u |_{I} \in C^\infty(I \times \TT^2) $. However, if one drops the control on the kinetic energy, one can obtain solutions with a small singular set in time (in the sense of Hausdorff dimension) using the gluing procedure in \cite[Section 3]{2009.06596}.

\end{enumerate}
\end{remark}

\subsubsection{Smoothness away from the bad set}
Finally, if we drop the control on the energy profile, we can arrange the solution to be close to any given smooth vector field outside of a bad set of an arbitrarily small measure.

\begin{theorem} \label{thm:Without_Energy_Profile}
Let  $u_0  $ be divergence free, $u_0 \in W^{s,1}(\TT^2)$ for some $s>-1$, and $ v \in C^\infty(\TT^2 \times [0,T])$ be a smooth divergence free vector fields.

Then for any $\ep>0$, there exist a weak solution with initial data $u(0) =u_0$, regularity
$$
u \in C((0,T]; W^{s',1}(\TT^2))  \quad\text{for all $s'<1$},
$$
and a Borel set $E \subset [0,T]$ with $\mathcal{L}^1([0,T] \setminus E )  < \ep$ 
such that 
\begin{equation}\label{eq:main_thm_Linfty_close}\tag{\(\dagger\)}
\| u(t)  - v(t)  \|_{L^\infty(\TT^2)} \leq \ep \quad \text{for all $t   \in  E$}.    
\end{equation}

\end{theorem}

\begin{remark}
We list a few remarks concerning Theorem \ref{thm:Without_Energy_Profile}.  

\begin{enumerate}

\item    The last property \eqref{eq:main_thm_Linfty_close} is a key feature of the temporal concentration mechanism: the solution is ``large'' only on a ``small'' set in time, cf. \cite{2004.09538,2009.06596}.

\item In fact, the measuring norm in \eqref{eq:main_thm_Linfty_close} does not have to be $L^\infty(\TT^2)$--it can be any Sobolev norm.
    
%\item By employing the gluing procedure in \cite{2009.06596}, one can also ensure the smoothness of $u$ on $E$. We refer to Section 3 of \cite{2009.06596} for details.
\end{enumerate}

\end{remark}

%%%%%%%%%%%%%%%%%%%%%%%%%%%%%%%%%%%%%%%%%%
\subsubsection{2D Euler equations and the Onsager conjecture}
%%%%%%%%%%%%%%%%%%%%%%%%%%%%%%%%%%%%%%%%%%

In general, if a weak solution $u$ of the  Euler equations belongs to the dimension independent Onsager's class $L^3([0,T];B^{1/3}_{3,c_0}$), then it conserves the kinetic energy \cite{MR2422377}. The $1/3$-H\"older regularity is known as the Onsager threshold and has been proven to be sharp in dimensions $d\geq 3$ \cite{MR3866888} but not yet in 2D.

However, in two dimensions there are additional regularization mechanisms that prevent anomalous dissipation, cf. \cite{MR3551263,MR4170311,MR4228012}. For instance, as shown in \cite{MR3551263}, weak solutions of the 2D Euler equations obtained as the vanishing viscosity limit of Leray-Hopf solutions of the 2D Navier-Stokes equations conserve energy provided $\omega_0 \in L^p$ for some $p>1$. We state a version of this result as Theorem~\ref{thm:2DEuler_energy_consrvation}, and demonstrate its sharpness in the class of weak solutions of the Navier-Stokes equation in the following theorem, which is a consequence of the main iteration scheme.

\begin{theorem}\label{thm:2deuler_vanishing_limit}
There exists a weak solution of the 2D Euler equations $u\in C([0,T];W^{s,1}(\TT^2))$ for all $s<1$ with $\curl u(0) \in L^p(\TT^2)$ for some $p>1$ that does not conserve energy, and a family $\{u^\nu\}$ of weak solutions to the 2D Navier-Stokes equations with
\[
\|u^\nu - u\|_{L^\infty(0,T; L^{p'}(\TT^2))} \to 0 \qquad \text{as} \quad \nu \to 0,
\]
for all $p'<2$.
\end{theorem}

%%%%%%%%%%%%%%%%%%%%%%%%%%%%%%%%%%%%%%%%%%
\subsection{Comments on the proof}
%%%%%%%%%%%%%%%%%%%%%%%%%%%%%%%%%%%%%%%%%%

Let us briefly explain the main ideas and difficulties of the proof. The main construction is based on a convex integration scheme that consists of adding highly oscillatory and concentrated perturbations to produce a weak solution of the 2D Navier-Stokes equations in the limit and an energy pumping mechanism that allows us to achieve the Lusin-type property for the kinetic energy profile. 

One of the biggest difficulties in any convex integration scheme designed for the Navier-Stokes equations is the presence of the Laplacian, which in 2D makes the equations $L^2$-critical. This is the very reason that we need to work in a regularity setting below this threshold. The framework of intermittent convex integration developed in \cite{MR3898708} only works above the regularity level $C_t L^2$ which is not sufficient for 2D Navier-Stokes equations.

The main building block in the convex integration scheme employed in this paper is a family of accelerating jet flows that are oscillatory and intermittent in space as well as  in time. These jet flows combine the advantages of the two constructions from \cite{1809.00600} and \cite{2009.06596}, and can be fully intermittent both in space and in time, in contrast to the ones in \cite{1809.00600} and \cite{2009.06596}. The main nonuniqueness results are the manifestation of the wild behaviors of such jet flows:
\begin{enumerate}
    \item Accelerating jet flows are at rest for the majority of the time  and only move on a union of small time intervals whose length models the strength of temporal intermittency;
    
    \item In space, these flows are divergence-free and supported in small moving balls. The extreme spatial intermittency  is needed for achieving the $C_t L^p$ regularity  for $p$ close to $2$. 
    
    \item On each time interval in the temporal supports, the jet flows first accelerate,  travel along the torus $\TT^2$ with large speed, and then come to a rest. They are periodic in both space and time.

    \item The jet flows solve the Navier-Stokes equations with a small $L^1_t W^{-1,1}$ error. In other words, the ``Reynolds stress'' of these jet flows is small in $L^1_{t,x}$.
\end{enumerate}

Once the main building blocks are obtained, we employ a space-time convex integration reducing the Reynolds stress to a highly oscillating in time error, which can be balanced by appropriate temporal and acceleration correctors, taking advantage of the space-time intermittency of the building blocks. To obtain nonuniqueness in the class $C_t L^p$ for $p<2$, we optimize the parameters in the accelerating flows so that they are fully intermittent in space but only slightly intermittent in time. The membership in this class and reducing the stress error at the same time require a very delicate choice of parameters in the accelerating flows. Once the accelerating flows are in the class $C_t L^p$, it is then fairly straightforward to obtain nonuniqueness in $\cap_{p<2} C_t L^p$ by a limiting argument in the iteration scheme, at least for smooth initial data.

To achieve nonuniqueness  for rough initial data, we take advantage of the smoothing property of the fractional heat flow. Choosing the power of Laplacian large enough, we show that any divergence-free initial data $u_0 \in W^{s,1}$ for some $s> -1 $ admits a solution $(u,R)$ of the Navier-Stokes-Reynolds system with the stress error  $R \in L^1_{x,t}$, which is compatible with our space-time convex integration scheme. Thanks to the smoothing property of the fractional heat flow, we can build a solution of the 2D Navier-Stokes equations on top of $(u,R)$, approaching the initial time as iterations progress.

Another key ingredient needed for the Lusin-type property is an energy pumping mechanism, done separately from the convex integration step, which essentially exploits the fact that the smallness in $L^\infty_t L^2$ falls in the subcritical regime of the 2D Navier-Stokes equations. Due to the time-intermittent nature of the convex integration scheme, the solution is mostly intact on a large portion of the time axis, what we call the good set $\mathcal{G}$, whose complement is of arbitrarily small measure. During the convex integration step, the perturbation can be made arbitrarily small on the good set $\mathcal{G}$ in any Sobolev space. On the other hand, we are able to correct the solution on the good set $\mathcal{G}$ to achieve desired properties, such as the $L^2$-continuity and gaining the prescribed energy level.

\subsection*{Organization}

The rest of the paper is organized as follows.
\begin{enumerate}
    \item In Section \ref{sec:outline}, we prove all the main theorems stated in the introduction assuming the main Proposition \ref{prop:main} of the convex integration iteration.
    \item Section \ref{sec:proof_main_prop} constitutes the first step of the proof of Proposition \ref{prop:main}---we construct the building blocks, accelerating jets, define the velocity perturbation, and decompose the resulting Reynolds stress.
    \item Section \ref{sec:proof_step_2} is devoted to estimating the velocity perturbation and the Reynolds stress error. We specify the oscillation, concentration, and acceleration parameters, and show that the inductive assumption of the iteration Proposition~\ref{prop:main} are satisfied. 
    \item  We show the the conservation of energy for vanishing viscosity solutions to the 2D Euler equations with initial vorticity in $L^p$, $p>1$ in Appendix~\ref{sec:append_zeroviscosity}. In Appendices~\ref{sec:append_improved_holder}, \ref{sec:append_antidiv} we recall some technical inequalities and properties of the antidivergence operators used in the convex integration scheme.  
\end{enumerate}

\subsection*{Acknowledgment}
AC was partially supported by the NSF grant DMS--1909849. The authors thank Manh Khang Huynh for pointing out an incorrect estimate of the energy corrector $\overline w_{n+1}$ in the  energy pumping step of  the proof of Theorem 1.2 in an earlier version of the paper.

%%%%%%%%%%%%%%%%%%%%%%%%%%%%%%%%%%%%%%%%%%%%%%%%%%%%%%%%%%%%%%%%%%%%%%%%%%%%%%%%%%%%%%%%%%%%%%%%%%%%%%%%%%%%%%%%%%%%%%%%%%%%%%
\section{Outline of the proof}\label{sec:outline}
%%%%%%%%%%%%%%%%%%%%%%%%%%%%%%%%%%%%%%%%%%%%%%%%%%%%%%%%%%%%%%%%%%%%%%%%%%%%%%%%%%%%%%%%%%%%%%%%%%%%%%%%%%%%%%%%%%%%%%%%%%%%%%

%%%%%%%%%%%%%%%%%%%%%%%%%%%%%%%%%%%%%%%%%%
\subsection{Notations}
%%%%%%%%%%%%%%%%%%%%%%%%%%%%%%%%%%%%%%%%%%
Throughout the paper, we fix our spatial domain $\TT^2 = \RR^2/\ZZ^2$ which we often identify with a periodic square $[0,1]^2$.

For any $d\in \NN+:=\NN\setminus{0}$, $\mathcal{L}^d$ denotes the Lebesgue measure in $\RR^d$.   For the Lebesgue spaces $L^p(\TT^2)$, the Lebesgue norms are denoted by
$$
 \|f \|_p = \|f\|_{L^p} = \|f \|_{L^p(\TT^2)} 
$$
where we do not distinguish scalar, vector or tensor valued functions. If the function $f:\TT^2 \times [0,T] \to \RR $ is time-dependent, we write
$$
\|f (t)\|_p  = \|f(t) \|_{L^p(\TT^2)}, 
$$
to indicate that the spatial norm is taken at each time slice $t\in [0,T]$.

The differentiation operations such as $\nabla$, $\Delta$, and $\D$ are meant for differentiation in space only, while for differentiation in time we use $\p_t$. 

In general, repeated indexed are summed without mentioning when the meaning is clear in the context. For matrix-valued function $A: \TT^2 \to \RR^{2\times 2}$, the usual (spatial) divergence is defined by
$$
\D A= \partial_j A_{ij}.
$$
For two vector valued functions $f ,g :\TT^2 \to \RR^2$, the tensor product $\otimes$ is given by $f\otimes g = f_i g_j :\TT^2 \to \RR^{2\times 2}$ and $ \mathring{\otimes}$ denotes the usual traceless tensor product $f \mathring{\otimes} g = f_i    g_j  -  \frac{1}{2}\delta_{ij}  f_k g_k$.

For any function $f$ on $\TT^2$, we say $f$ is $\sigma^{-1}\TT^2$-periodic if 
$$
f(x + \sigma^{-1} k) = f(x) \quad \text{for any $k\in \ZZ^2$},
$$
and for any $\sigma \in \NN$, we use the notation $f(\sigma \cdot )$ for the function $x \mapsto f(\sigma x)$, so $f(\sigma \cdot )$ is at least $\sigma^{-1} \TT^2$-periodic.

We use intensively the notation $X \lesssim Y$ which means $X \leq C Y$ for some constant $C >0$. Sometimes the constant will depend on various parameters and in that case we will write $X \lesssim_{\alpha, \beta, \gamma \dots} Y$ to indicate $X \leq C Y$ for some constant $C >0$ depending on $\alpha, \beta,\gamma \dots$. The notation $\gtrsim  $ is similarly defined and $X \sim Y$ means both $X \lesssim Y$ and $X \gtrsim Y$ at the same time.

%%%%%%%%%%%%%%%%%%%%%%%%%%%%%%%%%%%%%%%%%%
\subsection{The Navier-Stokes-Reynolds system}
%%%%%%%%%%%%%%%%%%%%%%%%%%%%%%%%%%%%%%%%%%
The proof of our main theorems is based on the construction of a sequence of approximate solutions. These approximate solutions (of \eqref{eq:NSE}) satisfy the so-called Navier-Stokes-Reynolds system:
\begin{equation}\label{eq:NSR}
\begin{cases}
\p_t u - \Delta u  + \D(u \otimes u) + \nabla p = \D R &\\
\D u =0,
\end{cases}
\end{equation}
where $R:\TT^2 \times [0,T] \to \mathcal{S}^{2\times 2}_0 $ is a $2\times 2$ symmetric traceless matrix, usually termed as the Reynolds stress in the literature. We refer to the discussion in \cite{MR4188806} for the motivation of this system and its relation to mathematical turbulence. In a nutshell, the Reynolds stress $R$ measures the approximateness of the solutions to exact solutions of the Navier-Stokes equations.

In what follows, we will say $(u,R)$ is a smooth solution to \eqref{eq:NSR} on $I$ if $u$ and $R$ satisfy \eqref{eq:NSR} and are both smooth on $ \TT^2 \times  I$ for some interval $I \subset \RR$. Note that the pressure is intentionally left out in this formulation since one can take the divergence of \eqref{eq:NSR} to obtain the following elliptic equation
$$
\Delta p = \D \D( R - u\otimes u),
$$
which uniquely determines the pressure when coupled with the usual zero mean condition $\int_{\TT^2} p  = 0$.

%%%%%%%%%%%%%%%%%%%%%%%%%%%%%%%%%%%%%%%%%%
\subsection{Universality of the Navier-Stokes-Reynolds flows}
%%%%%%%%%%%%%%%%%%%%%%%%%%%%%%%%%%%%%%%%%%
We prove the following auxiliary result used in the proof of Theorem \ref{thm:Energy_Profile_short} concerning the existence of (weak) solutions of \eqref{eq:NSR} for general rough data. 

Throughout the paper, we say $(u,R)$ is a weak solution of \eqref{eq:NSR} on $[0,T]$ if
\begin{equation}\label{eq:weak_formulation_u_R}
\int_{\TT^2} u_0(x) \varphi(  x, 0 ) \, dx = -   \int_{ \TT^2\times[0,1]} (u \cdot \Delta \varphi + u \otimes u  : \nabla \varphi + u  \cdot \partial_t \varphi ) \, dx dt  - \int_{ \TT^2 \times[0,1]} R : \nabla \varphi  \, dx dt  
\end{equation}
holds for all test functions $ \varphi\in \mathcal{D}_T  $ (smooth, divergence-free,  and vanishing on $ t\geq T$ ).

\begin{theorem}\label{thm:universality}

For any divergence-free $u_0 \in W^{s,p}(\TT^2) $, $s>-1$, $1 \leq p \leq \infty $, there exists a weak solution $(u , R) $ of \eqref{eq:NSR} on $ \TT^2 \times [0,T]$ such that $(u,R) $ is smooth on $ (0,T]$ and
$$
u\in C([0,T]; W^{s,p}(\TT^2)) \quad \text{and}\quad R \in L^1 (  \TT^2 \times  [0,T]).
$$
In fact, the velocity can be chosen to be $u(t) = e^{-\nu t (-\Delta)^\alpha} u_0$ for any $\alpha \geq 2 $ and $\nu>0$.
 
\end{theorem}
\begin{proof}
Take $u(t) = e^{-\nu t (-\Delta)^\alpha} u_0$ for some $\alpha \geq  2$, $\nu>0$, and define
\begin{equation}\label{eq:def_R_universality}
R : = \mathcal{R}(\p_t u - \Delta u) + u\mathring{\otimes } u,
\end{equation}
where $\mathcal{R}$ is the antidivergence defined in Appendix~\ref{sec:append_antidiv}. Note that the smoothness of $R$ on $\TT^2 \times ( 0,T]$ follows from that of $u$.

To prove the claims, we estimate each term in \eqref{eq:def_R_universality}. For any $0< t \leq 1$, by \eqref{eq:appendix_R_2}, $L^q$ boundedness of the Riesz transform for some $q>1$, and well-known estimates for the fractional heat semigroup, we have
\[
\begin{split}
\|\mathcal{R}  \p_t u  (t )\|_{L^1} = \|\mathcal{R} (-\Delta)^{\alpha} u (t )\|_{L^1}   
&\lesssim   \|(-\Delta)^{\alpha-\frac{1}{2}} u (t )\|_{L^q}\\
&    \lesssim t^{-1+\frac{1}{\alpha}(\frac{1}{q}-\frac{1-s}{2})} \||\nabla|^{s}u_0\|_{L^1},
\end{split}
\]
which is integrable by choosing $q $ sufficiently close to 1 when $\alpha>0$ and $s >-1$;
and
\[
\|\mathcal{R} \Delta u (t ) \|_{L^1} \lesssim \| |\nabla| u (t ) \|_{L^q}  \lesssim t^{\frac{1}{2\alpha}(s-1+ 2(1- \frac{1}{q}) ) } \||\nabla|^su_0\|_{L^1},
\]
which is integrable by choosing $q $ sufficiently close to 1 when   $\alpha \geq 1 $ and $s >-1$;
as well as
\[
 \| u\mathring{\otimes } u (t) \|_{L^1}   \lesssim \|u (t ) \|_{L^2}^2    \lesssim t^{\frac{s-1}{\alpha}} \||\nabla|^{s}u_0\|_{L^1}^2,
\]
which is integrable  when   $\alpha \geq 2$ and $s >-1$;

\end{proof}

%%%%%%%%%%%%%%%%%%%%%%%%%%%%%%%%%%%%%%%%%%
\subsection{The main iteration proposition}
%%%%%%%%%%%%%%%%%%%%%%%%%%%%%%%%%%%%%%%%%%
 
We now state the main proposition of the paper.

\begin{proposition}\label{prop:main}
There exist a geometric constant $M>0$ such that for any $1 \leq p < 2$, there exists $r>1$ depending only on $p$ with the following properties.

Let $(u,R)$ be a weak solution of \eqref{eq:NSR} on $  \TT^2 \times  [0,1]$ such that $(u,R)$ is smooth on $(0,1]$. Given any $\delta>0$ and any closed interval $I \subset (0,1]$, there exists another weak solution $(u_1,R_{1})$  of \eqref{eq:NSR} on $ \TT^2  \times   [0,1] $ and an open set $E \subset I $ such that 
\begin{enumerate}
    \item The velocity perturbation $w : = u_{1}  - u $ is smooth on $[0,1]$,
    $$
    \Supp_t w \subset I,
    $$
    and 
    \begin{align}
    \label{eq:Main_Iteration_Estimates_E_density}
    \mathcal{L}^1\big( [0,t] \cap E \big) \leq t\delta, \qquad  \text{for all $0\leq t \leq 1$.}
    \end{align}
    
    \item For any $t \not \in  E$
    \begin{align} \label{eq:Main_Iteration_Estimates_E^c}
     \|w  (t)\|_{W^{\frac{1}{\delta}, \infty} (\TT^2) } & \leq \delta,\\ \Big| \|u_1  (t)\|_{L^2(\TT^2)}^2  - \|u   (t)\|_{L^2(\TT^2)}^2 \Big| & \leq \delta^2 \label{eq:Main_Iteration_Estimates_E^c_2}.
    \end{align}

    \item There hold  the estimates,
\begin{align}
 \|w \|_{L^2(  \TT^2 \times  [0,1])} &\leq M  \| R \|^{\frac{1}{2}}_{L^1( \TT^2 \times I )}, \label{eq:main_prop_L2}\\   
 \|w \|_{L^{\infty}(0,1;   L^p (\TT^2))} & \leq    \delta\label{eq:main_prop_Lp}.
\end{align}

\item The new stress error $R_1$ satisfies
 $$
\|R_{1}  \|_{L^1(I; L^r( \TT^2  ))} \leq \delta .
$$
\end{enumerate}

\end{proposition}

In the above proposition, the role of the closed interval $I \subset (0,1]$ is to progressively get to $t=0$ along the iterations in the proof. For the current paper, one should think of $I$ as approximations of $(0,1]$. It also helps us to achieve the exact energy profile of the solution near $t=0$.

\begin{remark} \label{rem:Main_Iterarion_Remark}
The velocity perturbation in the above proposition enjoys the following additional properties.
\begin{enumerate}
\item There exists disjoint open intervals $I_n \subset I$ such that $E\subset \cup_n I_n$ and a small $r>0$, which goes to zero as  $\delta \to 0$, such that each $I_n$ in $\cup_n I_n $ is of length at most $r$, and the number of intervals $I_n$ is bounded by $\lesssim r^{-\frac{1}{9}}$.

\item In fact, thanks to Corollary~\ref{cor:L^inftyW^{s_p,1}}, estimate \eqref{eq:main_prop_Lp} can be replaced by $$
\|w \|_{L^{\infty}(0,1;   W^{s_p,1} (\TT^2))} 
\leq \delta,
$$ 
for some $0<s_p<1$ and $s_p \to 1$ as $p \to 2$. This is used to prove (sharp) nonuniqueness in Sobolev spaces $L^\infty_t W^{s,1}$, $s<1$, in Theorem~\ref{thm:KT_sharpness}.
\end{enumerate}

\end{remark}

%%%%%%%%%%%%%%%%%%%%%%%%%%%%%%%%%%%%%%%%%%%%%%%%%%%%%%%%%%%%%%%%%%%%%%%%%%%
\subsection{Proof of main nonuniqueness results} \label{sec:Proof_of_main_nonuniqueness_results}
%%%%%%%%%%%%%%%%%%%%%%%%%%%%%%%%%%%%%%%%%%%%%%%%%%%%%%%%%%%%%%%%%%%%%%%%%%%
We prove Theorems~\ref{thm:Energy_Profile_short}, \ref{thm:Main-L^2}, and \ref{thm:KT_sharpness} assuming Proposition \ref{prop:main}. Without loss of generality we assume $T=1$ and $\varepsilon < \frac{1}{2}$.

\begin{proof}[Proof of Theorems~\ref{thm:Energy_Profile_short}, \ref{thm:Main-L^2}, and \ref{thm:KT_sharpness}]
For $n \in \NN^+$, we construct a sequence of solutions $(u_n ,R_n)$, positive numbers $\delta_{n } \to 0$, monotone decreasing $t_n \to 0$ as $n \to \infty$, and open sets $E_n \subset [0,1]$ such that $(u_n ,R_n)$ converges to the desired weak solution $u$, and the union of $ E_n$ gives $ [0,1] \setminus \mathcal{G}$.

The solutions $(u_n , R_n)$  and sets $E_n$, constructed by induction, will satisfy the following assumptions.

\begin{enumerate}[label=(\alph*)]
    \item \label{item:uR_1} $(u_n ,R_n)$ is a smooth solution of \eqref{eq:NSR} on $\TT^2 \times (0,1] $;
  
    \item \label{item:uR_2} The Reynolds stress is small:
    \begin{align*}
    \| R_{n }\|_{L^1( \TT^2 \times [0,1] )} \leq \delta_{n};
    \end{align*}

    \item \label{item:uR_3} $u_n$ achieves the prescribed the energy level:
    \begin{align*}
      e(t) - \|u_n(t)  \|_2^2 & >  0 \quad \text{for all $ t   \in (0,1] \setminus \mathcal{B}_n$},\\
       e(t) - \|u_n(t)  \|_2^2 &\leq \frac{\delta_n^2 }{100} \quad \text{for all $ t   \in [t_n,1] \setminus \mathcal{B}_n$},
    \end{align*}
    where we  define the bad set $\mathcal{B}_n \subset [0,1]$ as the union of all $E_j$, $j \leq n$:
$$
\mathcal{B}_n = \bigcup_{1\leq j \leq n} E_j ;
$$
    \item \label{item:uR_4} The sets $E_n$ are unions of open intervals and satisfy
    $$
    \mathcal{L}^1([0,t] \cap E_n) \leq t\delta_n, \qquad \forall \, 0\leq t \leq 1.
    $$

\end{enumerate}   
Here, the parameters $p_n$ and $\delta_n$ are given explicitly as
\[
\qquad p_n= 2-2^{-n-1} \quad \text{for $n \in \NN^+$ }, \quad \text{and} \quad \delta_n = \ep 2^{-n} \leq 1 \quad  \text{if $n \geq 2$}.
\]
The value of $\delta_1>0$ and the sequence $t_n$ are chosen separately after we specify the initial solution $(u_1, R_1)$ in Step 1 below.

In addition, the perturbation $w_n: =u_{n}  -u_{n-1}$ satisfies for $n\geq 2$:
\begin{enumerate}[label=(\roman*)]
    \item  \label{item:w_1}   
    Regularity estimates:
    \begin{align*}
    \| w_n\|_{L^2( \TT^2 \times [0,1] )} &\leq C_M   \delta_{n-1}^{\frac{1}{2}} ,  \\
    \| w_n\|_{L^\infty(   0,1 ;L^{p_n}  )}& \leq   \delta_{n-1}, 
    \end{align*}
    where $  C_M>0 $ depends only on the initial solution $(u_1,R_1)$ and
    the geometric constant $M$ from Proposition~\ref{prop:main};
    \item \label{item:w_2}
    Support away from the origin:
    \begin{align*}
    \Supp w_n &\subset (0,1];
    \end{align*}
    \item \label{item:w_3} Small energy outside of $E_n$:
    \[ 
   \|w_n\|_{L^2(\TT^2)}        \leq \delta_n, \qquad t\in [t_n, 1] \setminus E_n.
\]
    
\end{enumerate}

\vspace{0.5em}
\noindent
{\bf Step 1: Initialization}

We choose the initial solution $(u_1,R_1)$ according to the given energy profile and $\ep>0$ in the main theorems. Take
\[
u_1(t) = e^{- \nu t (-\Delta)^{2}} u_0,
\]
for some $\nu>0$ that will be fixed large enough depending on $e(t)$. Thanks to
Theorem~\ref{thm:universality}, there exists a Reynolds stress $R_1 \in L^1(\TT^2\times [0,1])$ such that $(u_1,R_1)$ is a smooth solution to the Navier-Stokes-Reynolds system \eqref{eq:NSR} on $(0,1]$. The associated pressure can be obtained by a routine computation.

Now we define $E_1=[0,\frac{\ep}{2}]$ if $u_0 \notin L^2(\TT^2)$ (recall that $\ep<\frac{1}{2}$), and $E_1=\emptyset$ when $u_0 \in L^2(\TT^2)$. Note that $\|u_1(t)\|_{L^2}$ is bounded on any compact subset of $(0,1]$, and, in the case $p=2$, we also have $\|u_1(t) \|_2^2 \leq \|u_0\|_2^2e^{-\nu t}$. Thus, in either case, we can choose $\nu$ large enough so that
\[
\| u_1 (t) \|_2^2 < e(t), \quad \text{for all} \ \ t\in (0,T]\setminus E_1.
\]

Let us point out that $u_1$ is responsible for the difference in final regularity of $u$ between the case $p<2$ and $p=2$.

Since $R_1 \in L^1(\TT^2\times [0,1])$, we can specify the value of $\delta_1>1$ as 
\begin{equation}\label{eq:def_delta_1}
\delta_1 = 1+ 100\max\big\{\| R_1 \|_{L^1(\TT^2\times [0,1])},  \sup_{t\in[0,1]\setminus E_1} \big[e(t)-\|u_1(t)\|_2^2\big]^{\frac{1}{2}} \big\},
\end{equation}
and choose a strictly decreasing sequence of   $\{ t_n\}_{n \geq 1}$,  such that $t_n \to 0$ and
\begin{subequations}
\begin{align}
 100 \|R_1\|_{L^1(\TT^2 \times [0,t_{n}])} &\leq \delta_n,\label{eq:def_t_n_1}\\
t_n^{\frac{1}{2}} \|e\|_{L^\infty (0,1)} &\leq \frac{1}{4}\delta_{n} .\label{eq:def_t_n_2}
\end{align}
\end{subequations}
It follows that $(u_1 ,R_1)$ satisfies all the conditions \ref{item:uR_1}--\ref{item:uR_4}.  We also remark that condition \eqref{eq:def_t_n_2} is to get the energy density statement at $t=0$, so it is only used for the case $p=2$, as when $p<2$, we have $0 \in E_1 \subset \mathcal{B}_n$.

In the step 2--4 below, we will construct the solution $(u_n, R_n)$ and the sets $E_n$ for $n \geq 2$ by induction. 

\vspace{0.5em}
\noindent
{\bf Step 2: Energy correction $(u_n, R_n) \mapsto (\overline{u}_n, \overline{R}_n)$}

In this step, our main objective is to correct the energy level of $(u_n, R_n) $ to obtain a new solution pair $(\overline{u}_n, \overline{R}_n)$ with a slightly larger $\overline{R}_n $. This is done by adding a small energy corrector $ \overline{w}_{n+1}$ to $u_n$.

In view of inductive assumptions, the correction $\overline{w}_{n+1} : = \overline{u}_n - {u}_n $ needs to improve the energy level of $u_n$ on both $[t_{n+1}, t_n]$ and $[t_n,1 ] \setminus \mathcal{B}_n$. On the set $[t_n,1 ] \setminus \mathcal{B}_n$, the energy is already quite close to $e(t)$ due to \ref{item:uR_3}, and hence $\overline{w}_{n+1} $ is small in $L^2$ in this region. On $[t_{n+1}, t_n]$, the correction might be larger than $\delta_n$ since it is the first time we correct the energy there, but due to $t_n \to 0$ as in \eqref{eq:def_t_n_2} the correction will be small in $L^2_t$ norm. These observations ensure the needed convergence of $\overline{w}_{n+1}$ whose estimates are verified in the next step.

Now we turn to the specific construction of $\overline{w}_{n+1}$. Throughout the iteration, we fix a function $\psi \in C^\infty_c(\RR^2) $ with $\Supp  \psi\subset (0,1)^2$ and rescale it by a large parameter $\mu>0$, still denoting the resulting function by  $\psi$, so that
\begin{equation}\label{eq:def_psi_energy_correction}
\begin{aligned}
\| \nabla \psi \|_{L^2(\TT^2)}  = 1, \qquad
\| \nabla^n \psi \|_{L^{q}(\TT^2)} &\lesssim \mu^{n-1} \mu^{2(\frac{1}{2}  - \frac{1}{q})} \quad \text{for all $1 \leq q\leq \infty$}    .
\end{aligned}
\end{equation}
We will use $\psi$ as the stream function on $\TT^2$ for the energy corrector $ \overline{w}_{n+1}$. The exact value of $\mu$ depends on many factors and will be specified later.
 
 Now we can introduce a smooth nonnegative energy profile $\rho: [0,1] \to \RR^+$  such that all of the following holds.
\begin{enumerate}
    
\item $\rho(t)$ pumps the exact amount of energy outside of the bad set:
\begin{equation} \label{eq:def_of_e_n}
[\rho(t)]^2 = 
\left(1-\frac{\delta_{n+1}^2}{600\|e\|_{L^\infty (0,1)}}\right)\left[  e(t)  - \|u_n(t) \|_2^2  \right] \quad   \text{if $t \in [t_{n+1},1]\setminus \mathcal{B}_n$};
\end{equation}

 \item  $\rho(t)$ is under control on $[0,1]$:
 \begin{equation}\label{eq:def_cutoff_energy}
 [\rho(t)]^2  <
\begin{cases}
 e (t)- \|u_n(t)\|_2^2  \quad &\text{ if $t\in [t_{n+2},t_{n+1}]  $} \\ 
  \big\| e (t)- \|u_n(t)\|_2^2 \big\|_{L^\infty([t_{n},1]\setminus \mathcal{B}_n )} \quad &\text{ if $t\in [ t_{n } , 1]  $};
 \end{cases}    
 \end{equation}

%\begin{equation} \label{eq:def_of_e_n}
%[\rho(t)]^2 = 
%\left(1-\frac{\delta_2^2}{100\delta_1^2}\right)\left[  e(t)  - %\|u_n(t) \|_2^2  \right] \quad   \text{if $t \in [0,1]\setminus %\mathcal{B}_n$};
%\end{equation}

% \item  $\rho(t)$ is under control on $[0,1]$:
% \begin{equation}\label{eq:def_cutoff_energy}
% [\rho(t)]^2  \leq
%\begin{cases}
%\frac{1}{2}  (e (t)- \|u_n(t)\|_2^2 ) \quad &\text{ if $t\in %[0,t_{n+1}]  $} \\ 
%  \big\| e (t)- \|u_n(t)\|_2^2 \big\|_{L^\infty([t_{n},1]\setminus %\mathcal{B}_n )} \quad \text{ if $t\in [ t_{n } , 1]  $};
%% \end{cases}    
% \end{equation} 

\item $\rho(t)$ is supported away from $t=0$:
\begin{equation}\label{eq:def_cutoff_energy_zero_near_0}
\rho(t) = 0, \quad \text{for all $t\in[0,t_{n+2}]$}.
\end{equation}
 
\end{enumerate}

Such a profile $\rho(t)$ exists thanks to the fact that $\mathcal{B}_n \subset [t_n,1] $ and the inductive assumption \ref{item:uR_3} , and we use it together with the stream function $\psi$ to define the energy corrector $\overline{w}_{n+1}$:
$$
\overline{w}_{n+1} =   \rho(t)  \nabla^{\perp} \psi.
$$
We remark that in \eqref{eq:def_cutoff_energy} we use $t_n$ to separate the two cases as $\rho(t)$ could be large on $[t_{n+1}, t_n ]$ when comparing to $\delta_n$, cf. \eqref{eq:size_of_e_n} below. The exact choice of $\rho(t) $ is not substantial to our purpose as long as it fulfills \eqref{eq:def_of_e_n}--\eqref{eq:def_cutoff_energy_zero_near_0}. Immediately, we have
\begin{equation}\label{eq:size_of_w_n}
\| \overline{w}_{n+1} (t)  \|_2  = \rho(t), \quad \text{for all $t\in [0,1]$}.
\end{equation}
In addition, by assumption \ref{item:uR_3} and \eqref{eq:def_cutoff_energy}, it follows that
\begin{equation}\label{eq:size_of_e_n}
0\leq \| \overline{w}_{n+1} (t)  \|_2 \leq \frac{1}{10}\delta_n, \quad \text{for all $t\in [t_n,1]$}.
\end{equation}

The new solution $ (\overline{u}_n , \overline{R}_n)$ will then be defined as
$$
\overline{u}_n : = u_n + \overline{w}_{n+1},
$$
and 
\begin{equation}\label{eq:def_overline_R_n}
\overline{R}_n: = R_n + \mathcal{R}(- \Delta \overline{w}_{n+1} + \p_t \overline w_{n+1})   +  (\overline{w}_{n+1} \mathring{\otimes}  u_n + u_n \mathring{\otimes}  \overline{w}_{n+1})  + (\overline{w}_{n+1}  \mathring{\otimes} \overline{w}_{n+1}  ).
\end{equation}

\vspace{0.5em}
\noindent
{\bf Step 3: Estimates of $ (\overline{u}_n, \overline{R}_n)$ and $\overline{w}_{n+1}$}

In this step we establish the following estimates that will be passed to the next step. The estimates that we need are:
\begin{align}\label{eq:estimates_overline_R_n}
\| \overline{R}_n \|_{L^1(\TT^2 \times [0,1] ) } \leq    C \delta_n,
\end{align}
where $C>0$ is independent of $n$, as well as
\begin{subequations}
\begin{align}  
   \|\overline{w}_{n+1} \|_{L^2(0,1; L^{2}( \TT^2 )  ) } & \leq  \frac{1}{2}\delta_{n}, \label{eq:estimates_overline_w_n_1} \\
  \|\overline{w}_{n+1} \|_{L^\infty(0,1; L^{p_n}( \TT^{2} )  ) } & \leq  \frac{1}{2}\delta_{n}, \label{eq:estimates_overline_w_n_2}   
\end{align}
\end{subequations}
and
\begin{subequations}
\begin{align}
  e(t)  - \|   \overline{u}_n \|_2^2 &>   0 \quad \text{for all $t\in (0,1] \setminus \mathcal{B}_n $},\label{eq:energy_level_energy_correction_1}\\
0< e(t)  - \|   \overline{u}_n \|_2^2 &\leq   \frac{1}{500}\delta^2_{n+1}   \quad \text{for all $t\in [t_{n+1},1] \setminus \mathcal{B}_n $} \label{eq:energy_level_energy_correction_2}.    
\end{align}
\end{subequations}

The easiest one is \eqref{eq:estimates_overline_w_n_2}, which follow directly from \eqref{eq:def_psi_energy_correction} and \eqref{eq:size_of_e_n} by taking the free parameter $\mu$ sufficiently large. For \eqref{eq:estimates_overline_w_n_1}, we consider the split
\begin{align*}
\|\overline{w}_{n+1} \|_{L^2(0,1; L^{2}( \TT^2 )  ) } \leq \|\overline{w}_{n+1} \|_{L^2(0,t_n; L^{2}( \TT^2 )  ) } +\|\overline{w}_{n+1} \|_{L^2(t_n,1; L^{2}( \TT^2 )  ) } .
\end{align*}
The bound for the first term follows from \eqref{eq:def_t_n_2},
the first part in \eqref{eq:def_cutoff_energy}, and \eqref{eq:def_cutoff_energy_zero_near_0}:
%and \eqref{eq:def_of_e_n} noticing that $[0,t_n]\setminus \mathcal{B}_n=\emptyset$.
$$
\|\overline{w}_{n+1} \|_{L^2(0,t_n; L^{2}( \TT^2 )  ) } \leq t_n^{\frac{1}{2}} \sup_{t\in [0,t_{n}]} \rho(t) \leq \frac{1}{4} \delta_n,
$$
while the the second term follows from \eqref{eq:size_of_e_n}:
$$
\|\overline{w}_{n+1} \|_{L^2(t_n,1; L^{2}( \TT^2 )  ) } \leq \frac{\delta_n}{10}.
$$

Next, let us examine the energy level of $\overline{u}_n = u_n + \overline{w}_{n+1} $. Note that due to \eqref{eq:def_psi_energy_correction}, for any $t \in  [t_{n+2},1 ]  $  we have
\begin{align}
\left| \langle u_n,\overline{w}_{n+1} \rangle\right| & = \left| \int_{\TT^2} u_n \cdot \overline{w}_{n+1} \, dx\right| \nonumber\\
 &\leq  \| u_n\|_{L^\infty(\TT^2 \times [t_{n+2},1 ] )}  {\rho(t)} \| \nabla^{\perp} \psi \|_{L^1(\TT^2 )}\nonumber\\
 & \leq C(u_n) {\rho(t)} \mu^{-1}, \label{eq:size_of_e_n_small_interaction}
\end{align}
which tends to $0$ uniformly in time as $\mu \to \infty$. With this, we consider the energy difference,
\begin{equation}\label{eq:size_of_e_n_small_interaction2}
\begin{split}
 e(t)  - \|   \overline{u}_n (t) \|_2^2 &= e(t)- \|u_n(t)\|_2^2 -\| \overline w_{n+1}(t)\|_2^2 - 2\langle u_n,\overline{w}_{n+1} \rangle.
\end{split}    
\end{equation}

Next, we will use \eqref{eq:size_of_e_n_small_interaction2} to prove \eqref{eq:energy_level_energy_correction_1} and \eqref{eq:energy_level_energy_correction_2}. Since $[t_{n+2},1 ] \setminus \mathcal{B}_n$ is closed in $[0,1]$, it follows from \ref{item:uR_3} that $e(t) - \|  {u}_n (t)\|_2^2 $ is bigger than some positive constant on that set. Now thanks to \eqref{eq:def_cutoff_energy_zero_near_0}, \eqref{eq:def_of_e_n}, and \eqref{eq:size_of_e_n_small_interaction} we may choose $ \mu>0$ sufficiently large such that \eqref{eq:energy_level_energy_correction_1} holds:
$$
 e(t)  - \|   \overline{u}_n (t) \|_2^2 >0 \quad \text{for all $t\in (0,1] \setminus \mathcal{B}_n $}.
$$
For \eqref{eq:energy_level_energy_correction_2}, by 
%the inductive assumption \ref{item:uR_3}, \eqref{eq:def_of_e_n} and 
\eqref{eq:def_of_e_n} and \eqref{eq:size_of_w_n}, we see that the main terms in \eqref{eq:size_of_e_n_small_interaction2} satisfies
\begin{align*}
e(t)- \|u_n(t)\|_2^2 -\| \overline w_{n+1}(t)\|_2^2 &=\frac{\delta_{n+1}^2}{600\|e\|_{L^\infty (0,1)}}\big[e(t) - \|  {u}_n (t)\|_2^2\big] \\
&<  \frac{1}{500} \delta_{n+1}^2  \quad \text{for all $t\in [t_{n+1},1] \setminus \mathcal{B}_n $}.
\end{align*}
Thanks to the strict inequality above, we can thus take $ \mu>0$ sufficiently large such that \eqref{eq:size_of_e_n_small_interaction} and \eqref{eq:size_of_e_n_small_interaction2} yield
$$
e(t)  - \|   \overline{u}_n(t) \|_2^2 \leq   \frac{1}{500}\delta^2_{n+1}   \quad \text{for all $t\in [t_{n+1},1] \setminus \mathcal{B}_n $},
$$
thereby establishing \eqref{eq:energy_level_energy_correction_2}.

The remaining task is to obtain the estimate of the new stress $\overline{R}_n$ defined in \eqref{eq:def_overline_R_n}. First of all, we see that by \eqref{eq:def_psi_energy_correction} and the H\"older inequality,
\begin{align*}
\| u_n \mathring{\otimes} \overline{w}_{n+1} \|_{L^1(\TT^2 \times [0,1] ) } 
& \leq  \|u_{n}\|_{L^\infty(\TT^2 \times [0,1])}\| \overline{w}_{n+1}\|_{L^1(\TT^2 \times [0,1])}\\
&\leq C(e,u_n) \mu^{-1}, \\
\|  \mathcal{R}(  \p_t \overline w_{n+1}) \|_{L^1(\TT^2 \times [0,1] ) } &\lesssim \| \p_t \rho \|_{L^\infty([0,1])} \|\mathcal{R} \nabla^{\perp} \psi \|_{L^q(\TT^2)}\\
& \leq    C(e,u_n,    q) \mu^{1  - \frac{2}{q}}   \quad \text{for any $1<q < 2$},
\end{align*}
and
\begin{align*}
\| \overline{w}_{n+1} \mathring{\otimes} \overline{w}_{n+1} \|_{L^1(\TT^2 \times [0,1] ) }  & \leq   4 \|\overline{w}_{n+1} \|_{L^2(0,1; L^2( \TT^2 )  ) }^2 .
\end{align*}
Since $\mu>0$ is a free parameter, choosing suitably $q<2$ and taking $\mu$ sufficiently large once again, we can ensure
\begin{align}\label{eq:energy_correction_Rdri}
\| u_n \mathring{\otimes} \overline{w}_{n+1} \|_{L^1(\TT^2 \times [0,1] ) }  + \|  \mathcal{R}(  \p_t \overline w_{n+1}) \|_{L^1(\TT^2 \times [0,1] ) } \leq \delta_n^2.
\end{align}
On the other hand,   \eqref{eq:estimates_overline_w_n_1} implies 
\begin{equation}\label{eq:energy_correction_Rosc}
\| \overline{w}_{n+1} \mathring{\otimes} \overline{w}_{n+1} \|_{L^1(\TT^2 \times [0,1] ) }    \leq    \delta_n^2.   
\end{equation}

Finally, for the term involving the Laplacian in \eqref{eq:def_overline_R_n}, by \eqref{eq:appendix_R_2} and \eqref{eq:def_psi_energy_correction} we have
\begin{align}\label{eq:energy_correction_Rlap}
\|  \mathcal{R}(  \Delta \overline w_{n+1}) \|_{L^1(\TT^2 \times [0,1] ) } &= 2\|    \nabla \overline w_{n+1} \|_{L^1(\TT^2 \times [0,1] ) } \nonumber \\
& \leq  C_{\psi} \left(\int_{0}^{t_n} \rho(t)    + \sup_{t\in[t_n,1]}\rho(t)  \right)  \leq C_{\psi} {\delta_n}, 
\end{align}
where in the last step we have used bounds \eqref{eq:def_cutoff_energy}. We note that $C_{\psi}$ is a constant depending only on the profile $\psi$.

Now we add up \eqref{eq:energy_correction_Rdri}--\eqref{eq:energy_correction_Rlap} and, using the inductive assumption \ref{item:uR_2}, obtain the desired \eqref{eq:estimates_overline_R_n}:  
\begin{align*} 
\| \overline{R}_n \|_{L^1(\TT^2 \times [0,1] ) } \leq \| {R}_n \|_{L^1(\TT^2 \times [0,1] ) } + 2\delta_n^2  + C_{\psi} {\delta_n} \leq C \delta_n,
\end{align*}
for some constant $C$ independent of $n$, but dependent on $\delta_1$ (which is larger than $1$) and hence dependent on $(u_1,R_1)$.

\vspace{0.5em}
\noindent
{\bf Step 4: Error reduction  $(\overline{u}_n, \overline{R}_n) \mapsto ( {u}_{n+1},  {R}_{n+1})$}

In this step, we use the main proposition of convex integration to reduce the stress error.

We apply Proposition \ref{prop:main}  for the solution pair $(\overline{u}_n, \overline{R}_n) $ with $I= [t_{n+1}, 1]$ and parameters 
$$
\delta = \frac{1}{100}\min\big\{  \delta_{n+1} ,  \min_{t\in [t_{n+1},1]\setminus \mathcal{B}_n} \big[e(t) - \|\overline{u}_n(t)\|_2^2\big]^{\frac{1}{2}} \big\}   \quad \text{and} \quad p=p_{n+1} .
$$
where $\delta>0$ due to \eqref{eq:energy_level_energy_correction_2} and the fact that $[t_{n+1},1 ] \setminus \mathcal{B}_n$ is closed. Let us denote the obtained solution by $( {u}_{n+1},  {R}_{n+1}) $, the obtained set by $E_{n+1}$, and the velocity perturbation ${u}_{n+1} - \overline{u}_n $ by $\widetilde{w}_{n+1} $. Note that from the proposition, we have $\Supp_t \widetilde{w}_{n+1} \subset [t_{n+1} , 1]$ and the bound
\begin{equation} \label{eq:support_of_w_n+1}
 \| \widetilde{w}_{n+1}(t) \|_2 \leq \delta_{n+1}/100 \quad \text{for all  $t\not\in E_{n+1}$}   
\end{equation}
Now by \eqref{eq:main_prop_L2}, \eqref{eq:main_prop_Lp}, and \eqref{eq:estimates_overline_R_n}
we can fix $C_M \geq \delta_1^{\frac{1}{2}}$ throughout the iteration, depending only on
the geometric constant $M$ and $(u_1,R_1)$, such that 
\begin{subequations}
\begin{align}
   \|\widetilde{w}_{n+1} \|_{L^2(  \TT^2 \times [0,1] )  ) } &\leq M\|\overline{R}_n \|^{\frac{1}{2}}_{L^1(\TT^2 \times [0,1] ) } \nonumber \\
   &\leq M\cdot C^{\frac{1}{2}} \delta_n^{\frac{1}{2}}\leq \frac{1}{2}C_M\delta_n^{\frac{1}{2}}, \label{eq:estimates_overline_w_n+1_1}\\
   \| \widetilde{w}_{n+1} \|_{L^\infty(0,1; L^{p_n}( \TT^2 )  ) } & \leq  \frac{1}{2} \delta_{n+1}\leq  \frac{1}{2} \delta_{n},\label{eq:estimates_overline_w_n+1_2}
\end{align}
\end{subequations}
and
\begin{align}\label{eq:new_R_n+1}
\| R_{n+1}  \|_{L^1(  \TT^2 \times [t_{n+1},1] )  ) } \leq   \frac{1}{2}\delta_{n+1}.
\end{align}

The pair $( {u}_{n+1},  {R}_{n+1}) $ solves equation \eqref{eq:NSR} on $[0,1]\times \TT^2$ and the total stress error satisfies
$$
\| R_{n+1}\|_{L^1(  \TT^2 \times [0,1] )  ) } \leq \| R_{n+1}\|_{L^1(  \TT^2 \times [0,t_{n+1}] )  ) } + \| R_{n+1}\|_{L^1(  \TT^2 \times [t_{n+1},1] )  ) },
$$
which by \eqref{eq:def_t_n_1} and \eqref{eq:new_R_n+1} reduces to 
$$
\| R_{n+1}\|_{L^1(  \TT^2 \times [0,1] )  ) } \leq  \delta_{n+1}.
$$

In addition, due to \eqref{eq:Main_Iteration_Estimates_E^c_2}, the energy level for any $t  \in  [t_{n+1},1] \setminus  E_{n+1} $ satisfies
    \begin{align}\label{eq:estimates_energy_overline_w_n+1}
    \Big|\|u_{n+1} (t)\|_2^2  - \| \overline{u}_n   (t)\|_2^2\Big|     \leq \frac{\delta^2_{n+1}}{100^2}.
    \end{align}

Then for the inductive assumption \ref{item:uR_3}, we first observe that $e(t) - \|u_{n+1} (t)\|_2^2     >0 $ on $t\in ( 0, t_{n+1}]$ by the fact that $\Supp_t \widetilde{w}_{n+1} \subset [t_{n+1}, 1]$ and \eqref{eq:energy_level_energy_correction_1}. The same $e(t) - \|u_{n+1} (t)\|_2^2     >0 $ holds also on $t\in [t_{n+1}, 1] \setminus \mathcal{B}_{n+1}$ thanks to the definition of $\delta$ and \eqref{eq:Main_Iteration_Estimates_E^c_2} in the proposition. These establish the first bound in \ref{item:uR_3}. Similarly, the second bound in follows from \eqref{eq:estimates_energy_overline_w_n+1} and \eqref{eq:energy_level_energy_correction_2}

\vspace{0.5em}
\noindent
{\bf Step 5: Verification of inductive assumptions}

From Step 3 and Step 4, we see that $(u_{n+1}  , R_{n+1})$ satisfies the listed items \ref{item:uR_1}--\ref{item:uR_4} at level $n+1$. To close the induction argument, we only need to verify items \ref{item:w_1}, \ref{item:w_2} and \ref{item:w_3} for the total perturbation 
$$
w_{n+1}:= u_{n+1} - u_{n} = \overline{w}_{n+1} + \widetilde{w}_{n+1}.
$$
Now, we can conclude that:
\begin{itemize}
    \item Item \ref{item:w_1} follows from \eqref{eq:estimates_overline_w_n_1}--\eqref{eq:estimates_overline_w_n_2} and \eqref{eq:estimates_overline_w_n+1_1}--\eqref{eq:estimates_overline_w_n+1_2} by definition of $C_M$;
    \item Item \ref{item:w_2} is a direct consequence of \eqref{eq:def_cutoff_energy_zero_near_0}  and the definition of $\widetilde{w}_{n+1}$;
    
    \item Item \ref{item:w_3} follows from  \eqref{eq:size_of_e_n} and \eqref{eq:support_of_w_n+1}.
\end{itemize}

\vspace{0.5em}
\noindent
{\bf Step 6: Passing to the limit $n\to \infty$} 

Finally, we will conclude that the constructed sequence $(u_n ,R_n)$ converges to a weak solution $u$ of \eqref{eq:NSE}, and $u$ satisfies all the listed properties in the statement of Theorem \ref{thm:Energy_Profile_short}.

We define the final weak solution $u$ as the point-wise in time limit of $u_n$ in $L^1(\TT^2)$:
\begin{equation}
u(t) : =\lim_{n \to \infty } u_n(t),
\end{equation}
which is well-defined due to the inductive assumption \ref{item:w_1}.
In other words,
\begin{equation}
u(t) = u_1(t) + \sum_{n=2}^\infty  {w}_n (t).
\end{equation}
Since we know that $u_1 \in C([0,1] ; L^{p }(\TT^2) )\cap C^\infty(\TT^2 \times (0,T]) $, to show the claimed regularity of $u$ it suffices to prove that
$$
\sum_{n=2}^\infty  {w}_n     \in C([0,1] ; L^{p'}(\TT^2) ),  \quad  \text{for any $ p'<2$}.
$$ 
Indeed, for each $p'<2$ we can find $N_{p'}\in \NN$ such that $p_n > p'$ provided $n\geq N_{p'}$. Then by the inductive estimate \ref{item:w_1},
$$
\lim_{N\to \infty}\sum_{n=N}^\infty \|  {w}_n\|_{L^\infty L^{p'}}    \leq \lim_{N\to \infty}\sum_{n=N}^\infty\delta_{n-1} = 0.
$$
Sobolev space $L^\infty_t W^{s,1}$ regularity in Theorem~\ref{thm:KT_sharpness} similarly follows from estimates in Remark~\ref{rem:Main_Iterarion_Remark}.

Next, we show that the weak formulation holds. Indeed, for any $\varphi \in \mathcal{D}_T$, using weak formulation of \eqref{eq:NSR} or integrating by parts we obtain
\begin{equation}\label{eq:weak_formulation_u_n_R_n}
\begin{aligned}
\int_{\TT^2} u_0(x) \varphi(  x, 0 ) \, dx &= -   \int_{ \TT^2\times[0,1]} (u_n\cdot \Delta \varphi + u_n\otimes u_n : \nabla \varphi + u_n \cdot \partial_t \varphi ) \, dx dt  \\
&\qquad - \int_{ \TT^2 \times[0,1]} R : \nabla \varphi  \, dx dt .
\end{aligned}
\end{equation}
Inductive assumptions \ref{item:uR_1}--\ref{item:uR_4} imply the convergence of all terms to their natural limits in \eqref{eq:weak_formulation_u_n_R_n}. So $ u$ is a weak solution on $\TT^2 \times [0,1]$.

\vspace{0.5em}
\noindent
{\bf Step 7: Final conclusion} 

The good set in Theorems~\ref{thm:Energy_Profile_short} and \ref{thm:Main-L^2} is defined as
\[
\mathcal{G} = [0,1]\setminus \bigcup_{j\geq 1}E_j.
\]
By inductive assumption \ref{item:uR_3}, the energy achieves the prescribed level on $\mathcal{G}$:
\[
\|u(t)\|_{L^2}^2 = e(t), \qquad \forall t \in \mathcal{G}.
\]
Now by inductive assumption \ref{item:uR_4} and definition of $\delta_n$,
\[
\mathcal{L}^1([0,1] \setminus \mathcal{G}) \leq \sum_{n\geq 1}    \mathcal{L}^1(E_n) \leq \frac{\varepsilon}{2} + \sum_{n\geq 2} \delta_n = \varepsilon,
\]
which implies the  Lusin-type property for the kinetic energy in Theorem~\ref{thm:Energy_Profile_short}.

In the case when $u_0 \in L^2(\TT^2)$, the set $E_1$ is empty by definition. Recall also  that $E_n \subset [t_n,1]$ for some sequence $t_n \to 0$. Hence, employing the  inductive assumption \ref{item:uR_4} again, we obtain
\[
t^{-1}\mathcal{L}^1([0,t] \setminus \mathcal{G}) \leq t^{-1}\sum_{n:t_n<t}    \mathcal{L}^1([0,t] \cap E_n) \leq t^{-1}\sum_{n:t_n<t} t\delta_n \to 0,
\]
as $n\to \infty$, since $t_n \to 0$. This establishes density property \eqref{eq:density_L_2} in Theorem~\ref{thm:Main-L^2}.

Finally, we prove the second statement in
Remark \ref{remark:density_roughtID} concerning the $L^2$-continuity of $u$. Define the bad set
\[
\mathcal{B} = \bigcap_{n\geq 1}\bigcup_{j\geq n} E_j.
\]
By continuity of the Lebesgue measure from above,
\[
\begin{split}
\mathcal{L}^1(\mathcal{B}) = \lim_{n \to \infty } \mathcal{L}^1\big(\bigcup_{j\geq n} E_j\big)
=\lim_{n \to \infty } \varepsilon 2^{1-n} =0.
\end{split}
\]
In fact, it follows from Remark~\ref{rem:Main_Iterarion_Remark} that the Hausdorff dimension of $\mathcal{B}$ is less than or equal to $\frac{1}{9}$. Note that
\[
[0,1] \setminus \mathcal{B} = \bigcup_{n\geq 1}  \mathcal{T}_n,
\quad \text{where} \quad 
\mathcal{T}_n = \{0\}\cup [t_n,1]\setminus \bigcup_{j\geq n} E_j.
\]
Clearly $\mathcal{T}_n$ are increasing, and by inductive assumption \ref{item:w_3},
\[
\|w_j(t)\|_{L^2(\TT^2)} \leq \delta_j, \qquad \forall t \in \mathcal{T}_n, \ j \geq n.
\]
Since  $\delta_j \to 0$, we have that $u$ continuously maps $\mathcal{T}_n \to L^2(\TT^2)$ for every $n \in \NN$,  which establishes the second item in
Remark \ref{remark:density_roughtID}.

\end{proof}

%%%%%%%%%%%%%%%%%%%%%%%%%%%%%%%%%%%%%%%%%%%%%%%%%%%%%%%%%%%%%%%%%%%%%%%%%%%
\subsection{Proof of other main results}
%%%%%%%%%%%%%%%%%%%%%%%%%%%%%%%%%%%%%%%%%%%%%%%%%%%%%%%%%%%%%%%%%%%%%%%%%%%

Here we sketch the minor modifications needed to prove Theorem \ref{thm:Without_Energy_Profile} and Theorem \ref{thm:2deuler_vanishing_limit}.

\begin{proof}[Proof of Theorem \ref{thm:Without_Energy_Profile}]
We can use the same iteration as in Subsection~\ref{sec:Proof_of_main_nonuniqueness_results},  but without the inductive items \ref{item:uR_3} and \ref{item:w_3} which are only related to the kinetic energy. 

Take the initial solution $u_1$ to be $v$ and compute the associated Reynolds stress $R_1$. The smallness of $u-v$ in any Sobolev norm is then a consequence \eqref{eq:Main_Iteration_Estimates_E^c}, which we can add to the inductive assumptions for $(u_n, R_n)$. To conclude, we just delete Step 2 of energy correction from the proof and pass to the limit as before.

\end{proof}

\begin{proof}[Proof of Theorem \ref{thm:2deuler_vanishing_limit}]
We use a standard approximation argument. First, we note that Proposition~\ref{prop:main} holds for the Euler equations, and hence we can construct a weak solution of the 2D Euler equations $u\in C([0,T];W^{s,1}(\TT^2))$ for all $s<1$ with $\curl u(0) \in L^p(\TT^2)$ for some $p>1$ that does not conserve energy.

Now for each value of the viscosity $0<\nu<1$ we  mollify the Euler solution $u$ at a lengthscale $\sim \nu$ to obtain a smooth vector field $v^{\nu} $ which we use as the initial solution to the Navier-Stokes-Reynolds system $u_1:=v^{\nu}$ with an appropriate Reynolds stress $R_1$. Omitting the energy correction step again, we can construct a solution of the Navier-Stokes equations $u^\nu$ with
\[
\|u^\nu - v^{\nu}\|_{L^\infty(0,T; L^{2-\nu}(\TT^2))} \leq \nu.
\]
Hence the family of Navier-Stokes solutions $u^\nu$ approaches $u$ in the desired spaces:
\[
\|u^\nu - u\|_{L^\infty(0,T; L^{p'}(\TT^2))}\leq \|u^\nu - v^\nu\|_{L^\infty L^{p'}}+ \|u^\nu - v^\nu\|_{L^\infty L^{p'}} \to 0 \qquad \text{as} \quad \nu \to 0,
\]
for all $p'<2$.

\end{proof}

%%%%%%%%%%%%%%%%%%%%%%%%%%%%%%%%%%%%%%%%%%%%%%%%%%%%%%%%%%%%%%%%%%%%%%%%%%%%%%%%%%%%%%%%%%%%%%%%%%%%%%%%%%%%%%%%%%%%%%%%%%%%%%
\section{Proof of main proposition: velocity perturbation}\label{sec:proof_main_prop}
%%%%%%%%%%%%%%%%%%%%%%%%%%%%%%%%%%%%%%%%%%%%%%%%%%%%%%%%%%%%%%%%%%%%%%%%%%%%%%%%%%%%%%%%%%%%%%%%%%%%%%%%%%%%%%%%%%%%%%%%%%%%%%

In this section, we initiate the proof Proposition \ref{prop:main}. Typical in any convex integration scheme, the goal is to design a suitable velocity perturbation $w$ so that the new solution 
$$
u_1 : = u  + w
$$
solves the equation \eqref{eq:NSR} with a much smaller Reynolds stress $R_1$.  

The main objective here is to introduce the necessary preparation to define the velocity perturbation and the associated Reynolds stress. The full estimates claimed in Proposition \ref{prop:main} will be verified in the next section.

\subsection{Building blocks in space}

The velocity perturbation will consist of suitable superpositions of a family of vector fields with coefficients depending smoothly on the given stress error $R$.  We recall a crucial lemma from \cite{MR3614753} which motivates the construction of our building blocks in the convex integration.

Recall that $ \mathcal{S}^{2 \times 2}_+$ denotes the set of positive definite symmetric $2\times 2$ matrices and $\ek : = \frac{k}{|k|}$ for any $k \in  \ZZ^2$. For any matrix $A \in \RR^{2\times 2}$, we use the norm $|A| := \sup_{|x|=1} Ax$ for definiteness (the choice is insignificant for our purpose).

\begin{lemma}\label{lemma:geometric}
Let $ B_{\frac{1}{2}}(\Id)$ denote the ball of radius $1/2$ centered at identity Id in $ \mathcal{S}^{2 \times 2}_+$. There exists a finite set $\L \subset \ZZ^2$ and smooth functions $\Gamma_k \in C^\infty({B}_{\frac{1}{2}}(\Id) )$ for any $k \in \L$ such that 
\begin{align*}
R = \sum_{k \in \L } \Gamma_k^2(R)  \ek   \otimes  \ek  \qquad \text{for all } \, R \in {B_{\frac{1}{2}}(\Id)}.
\end{align*}

\end{lemma}

To utilize Lemma \ref{lemma:geometric}, we need to choose proper cutoffs functions for the stress error $R $. Let $\chi: \RR^{2\times 2}  \to \RR^+$ be a positive smooth function such that $\chi $ is monotone increasing with respect to $|x|$ and
\begin{equation}
\chi (x) =
\begin{cases}
2 & \text{if $0 \leq |x| \leq \frac{1}{2}\| R\|_{L^1( \TT^2 \times I )}$}\\
2|x| & \text{if $|x| \geq \| R\|_{L^1( \TT^2 \times I )}$}.
\end{cases}
\end{equation}

Let $\rho \in  C^\infty(\TT^2  \times  [0,1] )$ be
\begin{equation*} 
\rho =  \chi( R )  .
\end{equation*}
Then we see that 
$$
\Id - \frac{ R }{\rho} \in B_{\frac{1}{2}}(\Id) \qquad \text{for any $(x,t) \in   \TT^2 \times  [0,1]$,}
$$
so we can use it as the arguments in the amplitude functions given by Lemma \ref{lemma:geometric}. In particular, we have the following useful decomposition
\begin{equation}
\rho \Id - R =\sum_{k \in \L } R_k,
\end{equation}
where the ``primitive'' stresses $R_k:\TT^2 \times [0,1] \to \mathcal{S}^{2\times 2}_+$ are given by
\begin{equation}\label{eq:def_R_k}
R_k:= \rho \Gamma_k^2\Big(\Id - \frac{  R }{\rho }\Big) \ek \otimes \ek.
\end{equation}

The main building block of the convex integration scheme is a family of periodic accelerating jet flows $\mathbf{W}_{k}:\TT^2\times [0,1] \to \RR^2 $. This construction builds on previous work of
Buckmaster, Colombo, and Vicol \cite{1809.00600} 
and our temporal intermittency framework used in \cite{2004.09538, 2009.06596}.

We first construct a family of autonomous vector fields $W_k$ and then use intermittent temporal phase shifts $\Phi_k:\TT^2\times [0,1] \to \TT^2$ to obtain the final jets $\bwk:=W_k\circ \Phi_k $.

\begin{theorem}[Stationary jets]\label{thm:main_thm_for_W_k}
Let $\L\subset \ZZ^2 $ be a finite set. There exists $\mu_0>0$ such that for any $\nu, \mu \in \RR $ with $\mu_0 <\nu \leq \mu $ the following holds.

For every $k\in \L $, there exist smooth vector fields $ W_{k}, W_{k}^{(c)} \in C^\infty_0(\TT^2, \RR^2)$ with the following properties.
\begin{enumerate}
\item Each $ W_{k}+ W_{k}^{(c)}$ is divergence-free on $\TT^2$, and there is a stream function $  \Psi_k \in C^\infty_0(\TT^2 )$ such that
\[
W_k + W_k^{(c)}= \nabla^\perp  \Psi_k,
\]
and for each $k \in \L$
\[
 \ek \parallel W_k \quad \text{and}\quad \ek \perp W_k^{(c)},
\]
where $\ek=\frac{k}{|k|}$.

 \item Each $W_k $ has disjoint support
     $$
    \Supp W_k  \cap \Supp {W}_{k'} = \emptyset \quad \text{if $k \neq k'$}.
    $$
and for any $ k\in \L$, the following identities hold
\begin{equation} \label{eq:(WtimesW)_self_interaction}
\fint_{\TT^2} W_k \otimes W_k = \ek  \otimes \ek, 
\end{equation}
\begin{equation} \label{eq:div(WtimesW)_identity}
 \D ( W_k \otimes W_k) = \ek  \cdot \nabla |W_k|^2 \ek. 
\end{equation}

\item
 For any $1\leq p \leq \infty$, the estimates
 \begin{equation}\label{eq:W_k_estimates}
\begin{aligned}
\| W_k \|_{L^p(\TT^2)} +  {\mu}^{-1} \| \nabla W_k \|_{L^p(\TT^2)} &\lesssim (\nu \mu)^{\frac{1}{2}  - \frac{1}{p}}, \\
\| W_k^{(c)} \|_{L^p(\TT^2)} +  {\mu}^{-1} \| \nabla W_k^{(c)} \|_{L^p(\TT^2)} & \lesssim \nu \mu^{-1}  (\nu \mu)^{\frac{1}{2}  - \frac{1}{p}},
\end{aligned}  
 \end{equation}
and 
\begin{equation} \label{eq:stream_Phi_k_estimates}
\| \Psi_k \|_{L^p(\TT^2)} \lesssim \mu^{-1}(\nu \mu)^{\frac{1}{2} - \frac1p},
\end{equation}
holds uniformly in $\nu, \mu$.

\end{enumerate}
\end{theorem}
\begin{proof}

We choose a collection of distinct points $p_{   k} \in [0,1]^2 $ for $k \in \L $ and a number $\mu_0 >0$ such that
$$
\bigcup_{k\in \L} B_{\frac{2 }{\mu_0} }(p_k) \subset [0,1]^2,
$$
where $B_{\frac{2 }{\mu_0} }(p_k)$ denotes the ball of radius $\frac{2 }{\mu_0} $ and center $p_k$ and
\begin{equation}\label{eq:ep_0_pi_pj}
\dist(p_i, p_{j}) >   \frac{2}{\mu_0}  \quad \text{if} \quad  i\neq j . 
\end{equation}
The points $p_i$ will be the centers of $ W_k$ and $ W_k^c $.

For $k \in \L $, let us introduce unit vectors $\ek, \ek^\perp \in \RR^2$  by
\begin{equation}
\ek = \frac{k}{|k|}, \qquad  \ek^\perp = \frac{1}{|k|}(-k_2, k_1 ) ,
\end{equation}
and their associated coordinates: for any $x \in \RR^2$,
\begin{align*}
x_k &= (x- p_k)\cdot \mathbf{e}_k, \\
y_k & = (x- p_k)\cdot \mathbf{e}_k^\perp.
\end{align*}

Now we choose compactly supported nontrivial $\varphi, \psi \in C_c^\infty((- \frac{1 }{\mu_0},\frac{1 }{\mu_0}))$ and define non-periodic potentials $\widetilde\Psi_k \in C^\infty_c(\RR^2)$ and vector fields $\widetilde W_k, \widetilde W_k^{(c)} \in C^\infty_c(\RR^2) $
\begin{align*}
\widetilde \Psi_k &= c \mu^{-1}(\nu\mu)^{\frac{1}{2}} \varphi(\nu x_k) \psi(\mu y_k),\\
\widetilde W_k &= -c(\nu\mu)^{\frac{1}{2}} \varphi(\nu x_k) \psi'(\mu y_k) \ek,\\
\widetilde W_k^{(c)} &= c \nu \mu^{-1}(\nu\mu)^{\frac{1}{2}} \varphi'(\nu x_k) \psi(\mu y_k) \ek^\perp,
\end{align*}
where $c>0$ is a normalizing constant such that \eqref{eq:(WtimesW)_self_interaction} holds. 
Periodizing, we obtain
\[
\Psi_k(x) = \sum_{n\in \ZZ^2} \widetilde \Psi_k(x+n), \quad W_k(x) = \sum_{n\in \ZZ^2}  \widetilde W_k(x+n), \quad W_k^{(c)}(x) = \sum_{n\in \ZZ^2}  \widetilde W_k^{(c)}(x+n),
\]
which will be used as periodic building blocks, as we can now identify them with corresponding functions on $\TT^2$:
\[
\Psi_k: \TT^2 \to \RR, \qquad W_k: \TT^2 \to \RR^2, \qquad W_k^{(c)}: \TT^2 \to \RR^2.
\]

Note that
\[
\begin{split}
\nabla^\perp  \Psi_k &= -\p_{y_k} \Psi_k \ek +  \p_{x_k} \Psi_k \ek^\perp\\
&=  W_k + W_k^{(c)},
\end{split}
\]
and the rest of the conclusions follow trivially by direct computations as well.

\end{proof}

\begin{remark}
It is clear that the derivative of $\Psi_k$ in the direction $\ek$ is of order $  \nu$ (rather than $  \mu$ for the full gradient):
\begin{equation} \label{eq:e_k_derivative_of_Psi}
\|(\ek \cdot \nabla)\Psi_k \|_{L^r(
\TT^2)} \lesssim  \nu \cdot \mu^{-1}(\nu \mu)^{\frac{1}{2} - \frac1r}.  
\end{equation}
This fact will be used in Lemma \ref{lemma:Linear_error} later.
\end{remark}

Apparently, $W_k$ or $W_k+ W_k^{(c)}$ are not approximate solutions of the Navier-Stokes (or Euler) equations due to \eqref{eq:div(WtimesW)_identity}. This error term can not be absorbed by a pressure gradient due to its anisotropic nature. However, this very identity \eqref{eq:div(WtimesW)_identity} allows for letting $W_k$ travel along the geodesics on $\TT^2$ so that \eqref{eq:div(WtimesW)_identity} can be balanced by the acceleration of a small corrector.

Letting $W_k$ move along the geodesics creates one addition problem in 2D: without a suitable discretization of the temporal velocity, different $W_k$ might collide with each other and thus create harmful interactions. We will use a temporal concentration mechanism to get around this issue.

\subsection{Phase shift by acceleration}\label{subsec:phaseshift}
Next, we introduce a simple method to avoid the collision of the support sets of different $W_k$. This is based on our previous construction of temporal concentration in \cite{2009.06596}. We concentrate each $W_k$ on disjoint time intervals so that they have disjoint supports in $\TT^2 \times [0,1]$. The size of those time intervals then determines the level of temporal concentration.

Now we turn to the specifics. Let us first choose temporal functions $ g_{k}$ and $h_k$ to oscillate the building blocks $ W_k$ intermittently in time.
Let $G \in C_c^\infty(0,1)$ be such that
\begin{equation}\label{eq:def_G_profile}
\int_{0}^1 G^2(t) \, dt =1, \qquad \int_{0}^1 G (t) \, dt =0.    
\end{equation}

For any $\kappa \geq 1$, we define $\tilde g_k : [0,1] \to \RR$ as the $1$-periodic extension of  $\kappa^{1/2} G (\kappa (t - t_k) )$, where $t_k$ are chosen so that $g_k$ have disjoint supports for different $k$. In other words,
\begin{equation} \label{eq:def_of_tilde_g_k}
\tilde g_k(t) =   \sum_{n \in \ZZ}\kappa^{1/2} G (n + \kappa (t - t_k) ).
\end{equation}
In the convex integration scheme below, we will also oscillate the velocity perturbation at a large frequency $\sigma \in \NN$. So we define
\begin{equation}\label{eq:def_g_k}
g_k(t) =  \tilde g_k(\sigma t).
\end{equation}
Note that $g_k$ concentrates on different small time intervals, and we have
\begin{equation} \label{eq:L^p_bpund_g_k}
\| g_k \|_{W^{n,p}([0,1])} \lesssim (\sigma \kappa)^n \kappa^{\frac{1}{2} - \frac{1}{p}}\quad \text{for all $1\leq p \leq \infty$}, \ n \in \mathbb{N}.
\end{equation}

\begin{remark} \label{rem:Disjoint_supports}
With $g_k$ defined, we will essentially use $g_k W_k $ as the ``building blocks'' for the convex integration. Due to the temporal concentration of $g_k$, all $g_k W_k$ have disjoint supports on $\TT^2 \times [0,1]$, and hence there will be no interference between each different $W_k$.
\end{remark}

For the corrector term that we will be using to balance \eqref{eq:div(WtimesW)_identity}, define $h_k :[0,1] \to \RR $ by
$$
h_k(t)  = \int_0^{\sigma t} (\tilde g_k^2(s) -1) \, ds.
$$
In view of the zero-mean condition for $g_k^2(t) -1$, these $h_k$ are $\sigma^{-1}$-periodic on $[0,1]$ and we have
\begin{equation} \label{eq:bound_on_h_k}
\| h_k \|_{L^\infty([0,1])} \leq 1
\end{equation}
uniformly in $\kappa$.
Moreover, we have the identity
\begin{equation} \label{eq:time_derivative_of_h_k}
 \p_t \left( \sigma^{-1} h_k\right) = g_k^2 -1,
\end{equation}
which will imply the smallness of the temporal oscillation corrector, cf \eqref{eq:def_w_t}.

\subsection{Accelerating jets}

Let us recall the five parameters used for the building blocks.
\begin{itemize}
    \item spatial concentrations $\mu \geq 1$ and $\nu \geq 1 $: these are reciprocals of the longitudinal and lateral dimension of the building blocks $W_k $. 
    
    \item Oscillation $\sigma \in \NN$: we use it to oscillate the ``building blocks'' $g_k W_k$ both in space and in time so that the stress error $R$ is canceled weakly.
    
    \item Temporal concentration $\kappa \geq 1$: this parameter models the concentration in time of the perturbation (through $g_k$), which is crucial for controlling the new stress error emanating from the Laplacian.

    \item Acceleration $\omega \geq 1$: this parameter represents the acceleration of the flow in the building block, we use it to define the phase shift in \eqref{eq:def_of_phi_k}.

\end{itemize}

To make our notation more compact, let us introduce the phase function $\Phi_k : \TT^2 \times [0,1] \to \TT^2$ defined by
\begin{equation}\label{eq:def_Phi_k}
\Phi_k : (x  ,  t   ) \mapsto \sigma x + \phi_k(  t) \ek,   
\end{equation}
where $\phi_k(t)$ is defined by the relation
\begin{equation} \label{eq:def_of_phi_k}
\phi_k'(t)= \omega  g_k( t).
\end{equation}
Note that due to the zero-mean condition in \eqref{eq:def_G_profile}, such $\phi_k$ always exists and for definiteness, we fix one of such choices throughout the construction.

By design, $\Phi_k$ is $\sigma^{-1}\TT^2$-periodic in space and $\sigma^{-1}$-periodic in time. Also, by definition of the phase shift \eqref{eq:def_of_phi_k}, we have the important identities
\begin{equation} \label{eq:Phi_k_identities}
\begin{aligned}
\nabla \left(  f \circ \Phi_k \right)
&= \sigma\nabla f \circ \Phi_k, \\
\p_t \left(  f \circ \Phi_k \right) &=  \phi_k' \left( \ek \cdot \nabla  f \right)\circ \Phi_k =\sigma^{-1}  \omega  g_k \ek \cdot \nabla\left(  f \circ \Phi_k\right).
\end{aligned}
\end{equation}

Now we will let the stationary flows $W_k$ travel along $\TT^2$ in time, relating the velocity of the moving support sets  to the intermittent oscillator $g_k$. More precisely, we define
\begin{equation} \label{eq:W_k_W_kcPhi_k_definitions}
\begin{aligned}
\bwk (x,t ) : = W_k \circ \Phi_k =   W_k (\sigma x + \phi_k(t) \ek),\\
\bwk^{(c)}(x,t ) : = W_k^{(c)} \circ \Phi_k = W_k^{(c)} (\sigma x + \phi_k(t) \ek), \\
\bpk(x,t )  : = \Psi_k \circ \Phi_k =  \Psi_k (\sigma x + \phi_k(t) \ek).
\end{aligned}
\end{equation}
Hence
\begin{equation} \label{eq:grad_perp_stream_function}
\sigma^{-1} \nabla^{\perp} \bpk = \bwk + \bwk^{(c)}.
\end{equation}
Also, thanks to \eqref{eq:Phi_k_identities}  and  \eqref{eq:div(WtimesW)_identity}, we have the important identities
\begin{equation}\label{eq:acc_balance_error}
\begin{split}
\p_t | \bwk |^2 \ek  
&=  \sigma^{-1} \omega g_k  \D(\bwk \otimes \bwk)
\end{split},
\end{equation}
and
\begin{equation} \label{eq:p_t_of_Psi}
\begin{split}
\p_t \bpk  
&=  \sigma^{-1} \omega g_k  (\ek \cdot \nabla  )\bpk,
\end{split}
\end{equation}

\subsection{The velocity perturbation}

We are in the position to define the velocity perturbation. In summary, the perturbation consists of three parts:
\begin{equation}
w = w^{(p)} + w^{(c)} + w^{(t)},
\end{equation}
where $w^{(p)} $ is the principle part, accounting for the main contribution in the nonlinear term, $w^{(c)}$ is an incompresiblity corrector, rectifying the divergence of $w^{(p)} $, and $w^{(t)}$ is a temporal corrector with zero divergence, balancing the errors introduced by \eqref{eq:div(WtimesW)_identity} and \eqref{eq:def_Phi_k} below.

Next, we choose a cutoff for the velocity perturbation so that it lives strictly within the interval $I$. Since the stress error might accumulate near the endpoints of $I$, the cutoff will be sufficiently sharp so that the ``leftover'' error is negligible.
Let $\theta \in C^\infty_c(\RR)$ be a smooth temporal cutoff function such that $\|\theta\|_{L^\infty} \leq 1$ and
\begin{equation} \label{eq:def_theta}
\theta(t) = 
\begin{cases}
1 & \text{if }\dist(t, I^c ) \geq  \frac{  \delta  }{ 8  \| R\|_{L^\infty( \TT^2  \times   I)}}  \\
0 & \text{if }\dist(t, I^c ) \leq   \frac{  \delta  }{ 16  \| R\|_{L^\infty(\TT^2  \times  I )}}   . 
\end{cases}
\end{equation}

In view of \eqref{eq:def_R_k}, the amplitude functions of the perturbation are given by
\begin{equation}\label{eq:def_a_k}
a_k =  \theta   \rho^{\nicefrac{1}{2}} \Gamma_k\Big(\Id - \frac{  R }{\rho }\Big).
\end{equation}

The principle part of the perturbation consists of super-positions of the building blocks $\bwk  $ oscillating with period $\sigma^{-1}$ on $\TT^2$ and traveling with a velocity $ \phi_k'(t)$.
 
\begin{equation}\label{eq:def_w_p}
w^{(p)} (x  ,  t   ): = \sum_{k \in \L } a_k(x  ,  t   ) g_k( t)  \bwk  .
\end{equation}
In what follows, we will omit the set of the summation indexes $k\in \L$ so that
$$
w^{(p)}   = \sum_{k   } a_k g_k \bwk  .
$$

Note that \eqref{eq:def_w_p} is not divergence-free. To fix this, we introduce a divergence-free corrector

\[
w^{(c)} (x  ,  t   ): = \sum_{k } a_k(x  ,  t   ) g_k(t) \bwk^{(c)}   +\sigma^{-1} \nabla^\perp a_k(x  ,  t   ) g_k(t) \bpk  .
\]
Then, thanks to \eqref{eq:grad_perp_stream_function}, we have
\begin{equation}\label{def_of_w^p+w^c}
\begin{split}
w^{(p)}   +w^{(c)}   &= \sigma^{-1}\sum_{k  } g_k(t) \left[a_k \nabla^\perp \bpk  +  \nabla^\perp a_k \bpk  \right]\\
&= \sigma^{-1}  \sum_{k  } \nabla^\perp \big[a_k g_k  \bpk \big].
\end{split}
\end{equation}

Finally, we define a temporal corrector $w^{(t)}$. The goal of this corrector $w^{(t)}$ is two-fold: $(i)$ to balance the high temporal frequency part of the stress error; and $(ii)$ to balance the error in the interaction due to the acceleration in the phase shifts $\Phi_k$.
 
To introduce these correctors, we first recall the well-known Leray projection on $\TT^2$.
\begin{definition}[Leray projection]\label{def:leray_projection}
Let $v \in C^\infty(\TT^2,\RR^2)$ be a smooth vector field. Define the operator $\mathcal{Q}$ as
$$
\mathcal{Q} v:= \nabla f + \fint_{\TT^2}v,
$$
where $ f \in C^\infty(\TT^2)$ is the unique smooth zero-mean solution of 
$$
\Delta f =\D v, \qquad x \in \TT^2.
$$
Furthermore, let $\mathcal{P} = \Id -\mathcal{Q}$ be the Leray projection onto divergence-free vector fields with zero mean.
\end{definition}

The temporal corrector $w^{(t)}$ is defined as
\begin{equation}\label{eq:def_w_t}
w^{(t)} = w^{(o)} + w^{(a)}, 
\end{equation}
where $w^{(o)}$ is the temporal oscillation corrector 
\begin{equation}
\begin{aligned} \label{eq:def_w_o}
w^{(o)} (x  ,  t   ) & :=- \sigma^{-1} \theta^2  \mathcal{P}\sum_{k\in \L }    h_k(t)   \D R_k     ,
\end{aligned}
\end{equation}
and $w^{(a)}$ is the acceleration corrector 
\begin{equation}\label{eq:def_w_a}
w^{(a)} (x  ,  t   ): = -\omega^{-1} \sigma  \mathcal{P}\sum_{k \in \L } a_k(x  ,  t   )^2 g_k(  t)   |\bwk|^2 \ek .
\end{equation}

Let us say a few words about the corrector $w^{(t)} $. The corrector $ w^{(t)}$ will be used to balance part of the new stress error through its temporal derivative $\p_t w^{(t)} $, but the roles of each part $ w^{(o)}$  and $w^{(a)} $ are very different. To see the leading order temporal derivative of $w^{(a)}$, 
thanks to fact that $\phi'(t)=\omega g(t)$ and due to \eqref{eq:acc_balance_error}, we have the heuristic
\[
\p_t w^{(a)}   = -\mathcal{P}\sum_{k } a_k^2g_k^2 \D \big(  \bwk \otimes \bwk   \big) +
\text{lower order terms},
\]
which is needed for a cancellation in the oscillation error. On the other hand, thanks to \eqref{eq:time_derivative_of_h_k},
$$
\p_t w^{(o)}     =   -\theta^2  \mathcal{P}  \Big(  \sum_{k\in \L }    (g_k^2 -1) \D R_k \Big)     + \text{lower order terms},
$$
which we use to cancel a highly time oscillating remainder of the Reynolds stress, specific to the use of temporal intermittency in our scheme.  We compute this in the following lemma.

\begin{lemma}\label{lemma:a_k_interactions}
For any $(x,t) \in \TT^2 \times [0,1]$, there holds
$$
 \sum_{k \in \L} a_k^2 g_k^2 \fint_{\TT^2}  \bwk \otimes \bwk \, dx = \theta^2  \left(   \rho \Id   - R
 \right)  + \theta^2  \sum_{k\in \L }   (g_{k}^2-1 ) R_k,
$$
where $R_k$ is defined by \eqref{eq:def_R_k}.
\end{lemma}
\begin{proof}
First we note that by Theorem \ref{thm:main_thm_for_W_k}, for any $t$, 
$$
\fint_{\TT^2}  \bwk (x,t) \otimes \bwk(x,t) \, dx = \fint_{\TT^2}  W_k(x) \otimes W_k (x)  \, dx= \ek\otimes \ek.
$$
So by Lemma \ref{lemma:geometric}, a direct computation gives
\begin{align*}
\sum_{k \in \L } a_k^2 g_k^2 \ek \otimes \ek &= \theta^2\rho \sum_{k  \in \L }  g_{k}^2  \Gamma_k^2\Big(\Id - \frac{R }{\rho }\Big) \ek  \otimes  \ek\\
&= \theta^2   \sum_{k  \in \L}R_k  + \theta^2   \sum_{k\in \L }   (g_{k}^2-1 ) R_k\\
&= \theta^2  \left(   \rho \Id   - R
 \right)  + \theta^2   \sum_{k\in \L }   (g_{k}^2-1 ) R_k.
\end{align*}

\end{proof}

\subsection{The new Reynolds stress} \label{sec:New_Reynolds_stress}
In this subsection, our goal is to design a suitable stress tensor $R_1 : \TT^2 \times [0,1]\to \mathcal{S}^{2\times2}_0 $ such that the pair $(u_1, R_1)$ is a smooth solution of \ref{eq:NSR} for a suitable smooth pressure $p_1$. 

We will define the new Reynolds stress by
$$
R_1 =  R_{\Lin} +R_{\Cor} + R_{\Osc} ,
$$
such that
\begin{align}
  \D R_{\Osc}   &= \p_t w^{(t)}  + \D(w^{(p)} \otimes w^{(p)} + R    ) + \nabla P, \label{eq:R_osc_intro}\\ 
 \D R_{\Cor} &=    \D \left( (w^{(c)}+w^{(t)})  \otimes w   + w^{(p)} \otimes (w^{(c)} + w^{(t)} )   \right),   \label{eq:R_cor_intro}\\
 \D R_{\Lin} &=   \p_t (w^{(p)} +w^{(c)}  ) - \Delta  u  +   \D\left( u  \otimes w + w \otimes   u \right)   \label{eq:R_lin_intro}.
\end{align}

To this end, we will use the antidivergence operator $ \mathcal{R} : C^\infty(\TT^2 ,\RR^2) \to C^\infty(\TT^2, \mathcal{S}^{2\times 2}_0)$ defined in Appendix \ref{sec:append_antidiv}. True to its name, this antidivergence operator $ \mathcal{R} $ satisfies
\begin{equation}\label{eq:R_antidiv_remove_mean}
\D ( \mathcal{R} v  ) = v - \fint_{\TT^2} v  \quad \text{for any $v \in C^\infty(\TT^2 ,\RR^2)$}.    
\end{equation}

Since \eqref{eq:R_osc_intro}-\eqref{eq:R_lin_intro} are time-dependent, it should be understood that the antidivergence  $ \mathcal{R}$ is applied for each time slice $t\in [0,1]$, and the resulting stresses are smooth on $\TT^2 \times [0,1]$.
 
In addition, we will also use a bilinear antidivergence operator $ \mathcal{B}$ in Appendix \ref{sec:append_antidiv}, which satisfies
\begin{align}\label{eq:B_antidiv_remove_mean}
\D( \mathcal{B}(v , A)  )   
 =     v A -       \fint v A,  \quad \text{for $v \in C^\infty(\TT^2,\RR^2)$ and $A \in C^\infty_0(\TT^2,\RR^{2\times 2})$},
\end{align}
where we denote $v A = v_j A_{ij}$ instead of the usual $A v$. In what follows the matrix $A$ is often symmetric, so there will be no confusion in writing $v A$ versus $A v$.

This bilinear antidivergence $ \mathcal{B}$ has the advantage of gaining derivative from the second argument $A$ when it only has high spatial frequencies. This has a flavor of the classical stationary phase and has been a useful tool for many estimates in convex integration \cite{arXiv:2007.08011}.

\subsection{Computing the oscillation error}
Our first step of solving for $R_1$ is to derive $R_{\Osc}$. Let us compute the nonlinear term $\D(w^{(p)} \otimes w^{(p)} +  R  )$. Due to \eqref{eq:def_of_tilde_g_k} and \eqref{eq:def_g_k} (see Remark~\ref{rem:Disjoint_supports}), we have
$$
\Supp g_k \bwk \cap \Supp g_{k'}\mathbf{W}_{k'}  =\emptyset\quad \text{if $k\neq k'$}.
$$
It follows that we only have self-interactions of the accelerating jets:
\[
\D(w^{(p)} \otimes w^{(p)} +  R  ) = \D \Big[ \sum_{k \in \L} a_k^2 g_k^2 \bwk \otimes \bwk  + R      \Big].
\]
We use Lemma \ref{lemma:a_k_interactions} to remove the leading order interaction given by the spatial mean of  $\bwk \otimes \bwk$ and obtain
\begin{equation} 
\begin{split}
\D(w^{(p)} \otimes w^{(p)} +  R  )   &  =  \D \sum_{k } a_k^2 g_k^2 \big( \bwk \otimes \bwk   -\fint  \bwk \otimes \bwk  \big)  +  \nabla ( \theta^2  \rho   )  + (1 -\theta^2    )\D  R \\
&\qquad  + \theta^2      \sum_k (g^2_k -1) \D R_k   . 
\end{split}
\end{equation}
 
Putting $\p_t w^{(t)}$ into action, using $\mathcal{P} = \Id - \mathcal{Q} $ and the product rule, we separate the terms into four different groups:
\begin{equation}\label{eq:New_Reynolds_Stress_01}
\begin{aligned}
\p_t w^{(t)} +  &\D(w^{(p)} \otimes w^{(p)} +  R  ) \\  &  =  \underbrace{ \sum_{k } g_k^2 \nabla(a_k^2 ) \cdot  \big( \bwk \otimes \bwk   -\fint  \bwk \otimes \bwk  \big)
+  \mathcal{Q} \sum_{k } a_k^2 g_k^2 \D( \bwk \otimes \bwk  ) }_{=:E_1} \\
& \quad + \underbrace{\p_t w^{(a)} +  \mathcal{P} \sum_{k } a_k^2 g_k^2 \D( \bwk \otimes \bwk  ) }_{=:E_2} \\
& \qquad  + \underbrace{ \p_t w^{(o)} + \theta^2   \sum_k (g^2_k-1) \D R_k   }_{=:E_3} \\
& \qquad\quad +\underbrace{ \nabla ( \theta^2  \rho   )  + (1 -\theta^2    )\D  R }_{=:E_4}. 
\end{aligned}
\end{equation}

Before we move on to the analysis of each term $E_i$, let us point out the following simple consequence of \eqref{eq:R_antidiv_remove_mean} and the definition of $\mathcal{P}$:
\begin{equation}\label{eq:P_R_pressure}
 \mathcal{P} f = \D (R f) - \nabla ( \Delta^{-1} \D f) \quad\text{for any $f \in C^\infty(\TT^2 , \RR^2)$},
\end{equation}
where we note that $f$ does not need to be mean-free.
We will use \eqref{eq:B_antidiv_remove_mean} and \eqref{eq:P_R_pressure} to re-write each $E_i$ as the sum of the divergence of a stress and the gradient of a pressure.

\noindent
{\bf Analysis of $E_1$:}

Recall that by Definition~\ref{def:leray_projection},
\begin{equation}\label{eq:Qv_definition}
\mathcal{Q} v= \nabla \Delta^{-1} \D v +\fint v,
\end{equation}
and consequently $\fint \mathcal{Q} v =\fint v$, which implies that $E_1$ has zero spatial mean.
%the following sum has zero spatial mean
%\begin{equation}
%\sum_{k }g_k^2 \nabla(a_k^2 ) \cdot  \big( \bwk \otimes \bwk    -\fint  \bwk \otimes \bwk % \big)
%+\mathcal{Q}\sum_{k  } a_k^2 g_k^2 \D \big(\bwk \otimes \bwk  \big).
%\end{equation}
Therefore, if we define 
\begin{equation}\label{eq:New_Reynolds_Stress_03}
R_{\Osc,x} =  \sum_{k \in \L}g_k^2 \mathcal{B}\Big(\nabla(a_k^2 ) ,    \bwk \otimes \bwk  -\fint \bwk \otimes \bwk \Big),
\end{equation} 
and the pressure  
\[
p_1 =  \Delta^{-1} \D \sum_{k \in \L } a_k^2 g_k^2 \D \big(\bwk \otimes \bwk \big),
\]
then by  \eqref{eq:B_antidiv_remove_mean} and \eqref{eq:Qv_definition} we have 
\begin{equation}
E_1 =\D R_{\Osc,x} + \nabla p_1,
\end{equation}
where the cancellation of the means follow from the fact that $\fint E_1 =0$.

\noindent
{\bf Analysis of $E_2$:}

We use \eqref{eq:def_w_a} and \eqref{eq:acc_balance_error} to compute  
$$
\p_t   w^{(a)}=  -\mathcal{P} \sum_{k } a_k^2 g_k^2 \D( \bwk \otimes \bwk  )  -\omega^{-1} \sigma  \mathcal{P}\sum_{k  } \p_t \left(a_k^2 g_k \right) |\bwk|^2 \ek.
$$
which implies that
\begin{equation} \label{eq:New_Reynolds_Stress_022}
\begin{split}
E_2 &=  -  \omega^{-1} \sigma  \mathcal{P}\sum_{k  } \p_t \left(a_k^2 g_k \right) |\bwk|^2 \ek.
\end{split}
\end{equation}
In view of \eqref{eq:P_R_pressure}, we define an approximation stress
\begin{equation}
R_{\Osc, a} = -  \omega^{-1} \sigma   \sum_{k \in \L } \mathcal{R}\left( \p_t \left(a_k^2 g_k \right) |\bwk  |^2 \ek \right)
\end{equation}
and a pressure term
$$
p_2= \omega^{-1} \sigma  \Delta^{-1} \D \sum_{k \in \L } \p_t \left( a_k^2 g_k \right) |\bwk  |^2 \ek,
$$
such that
$$
E_2 = \D R_{\Osc, a} + \nabla p_2.
$$

\noindent
{\bf Analysis of $E_3$:}

For this term, let us first compute the time derivative of $w^{(o)}$ \eqref{eq:def_w_o} using \eqref{eq:time_derivative_of_h_k}:
$$
\p_t w^{(o)}     =   -\theta^2  \mathcal{P}  \Big(  \sum_{k\in \L }    (g_k^2 -1) \D R_k \Big)    - \sigma^{-1}\theta^2     \mathcal{P}\sum_{k\in \L }     h_k   \D \p_t R_k     .
$$ 
Thus
\begin{align*}
E_3 &=  \p_t w^{(o)} + \theta^2  \sum_k (g^2_k-1) \D R_k \\
&=  \theta^2 \mathcal{Q}    \sum_k (g^2_k-1) \D R_k     - \sigma^{-1}\theta^2  \mathcal{P}\sum_{k\in \L }     h_k     \D \p_t R_k.
\end{align*}

Using definition of $\mathcal{Q}$ and \eqref{eq:P_R_pressure} once again, we define
\begin{equation}
p_3 = \theta^2  \sum_k (g^2_k-1) \Delta^{-1} \D \D R_k +
\sigma^{-1}\theta^2 \sum_{k\in \L }     h_k     \Delta^{-1}\D \D \p_t R_k   
\end{equation}
and
\begin{equation}\label{eq:def_osc_t_error}
R_{\Osc,t} = - \sigma^{-1}\theta^2  \mathcal{R}\sum_{k\in \L }     h_k      \D \p_t R_k    
\end{equation}
to obtain
\begin{align*}
E_3 &=   \D R_{\Osc,t} + \nabla p_3.
\end{align*}

\noindent
{\bf Analysis of $E_4$:}

We just define
\begin{equation}
p_4 = \theta^2 \rho,
\end{equation}
and
$$
R_{\Rem } =  (1 -\theta^2)      R. 
$$
Then
\begin{align*}
E_4 &=   \D R_{\Rem }  + \nabla p_4.
\end{align*}

Combining all the terms, we have thus proven:
\begin{lemma}\label{lemma:R_osc_decomposition}
Define the oscillation error $R_{\Osc}$ 
$$
R_{\Osc} = R_{\Osc,x}+ R_{\Osc,t} + R_{\Osc, a}+ R_{\Rem}.
$$
and the pressure $P:= -p_1-p_2 -p_3 -p_4 $.
Then
$$
\p_t w^{(t)}  + \D(w^{(p)} \otimes w^{(p)} + R    ) + \nabla P = \D R_{\Osc}.
$$
\end{lemma}

\subsection{Finalizing the new solution}
Finally, we define the correction error and the linear error in the usual way:
\begin{equation}\label{eq:def_R_cor}
R_{\Cor} =  \mathcal{R} \D \big(  (w^{(c)}+w^{(t)})  \otimes w   + w^{(p)} \otimes (w^{(c)} + w^{(t)} )\big),
\end{equation}
and
\begin{equation}\label{eq:def_R_lin}
R_{\Lin} = \mathcal{R}\big(  \p_t (w^{(p)} +w^{(c)}  ) - \Delta  u  +   \D( u  \otimes w) +\D( w \otimes   u  ) \big).
\end{equation}

 Now we can conclude the construction of the new solution to the Navier-Stokes-Reynolds system.

\begin{lemma}\label{lemma:new_stress_R_1}
Define the new Reynolds stress by
$$
R_1 =  R_{\Lin} +R_{\Cor} + R_{\Osc} ,
$$
and the new pressure by
$$
p_1 =    p  +P.
$$
Then $(u_1,R_1)$ is a solution to \eqref{eq:NSR},
$$
\p_t u_1 -\Delta u_1 + \D(u_1 \otimes u_1 ) + \nabla p_1 = \D R_1.
$$
\end{lemma}
\begin{proof}
Since $( u  ,  R )  $ solves \eqref{eq:NSR} with pressure $p$, a direct computation gives
\begin{align*}
\p_t u_1 -\Delta u_1 + &\D(u_1 \otimes u_1 ) + \nabla p_1  \\
=& \p_t  u  -\Delta  u   + \D(  u  \otimes  u  ) + \nabla p\\
& +  \p_t w -\Delta w + \D( u   \otimes w ) + \D(w  \otimes  u  ) +\D(w  \otimes w ) + \nabla P  \\
= & \D  R +  \p_t w -\Delta w + \D( u   \otimes w ) + \D(w  \otimes  u  ) +\D(w  \otimes w ) + \nabla P.
\end{align*}
Now by Lemma \ref{lemma:R_osc_decomposition} and definitions  \eqref{eq:def_R_cor}, \eqref{eq:def_R_lin}  we conclude that $(u_1, R_1)$ solves \eqref{eq:NSR}.
\end{proof}

%%%%%%%%%%%%%%%%%%%%%%%%%%%%%%%%%%%%%%%%%%%%%%%%%%%%%%%%%%%%%%%%%%%%%%%%%%%%%%%%%%%%%%%%%%%%%%%%%%%%%%%%%%%%%%%%%%%%%%%%%%%%%%
\section{Estimates for the velocity and the stress error}\label{sec:proof_step_2}
%%%%%%%%%%%%%%%%%%%%%%%%%%%%%%%%%%%%%%%%%%%%%%%%%%%%%%%%%%%%%%%%%%%%%%%%%%%%%%%%%%%%%%%%%%%%%%%%%%%%%%%%%%%%%%%%%%%%%%%%%%%%%%

We will show in this section that the velocity perturbation $w$ and the new Reynolds stress $R_1$ derived in Section \ref{sec:New_Reynolds_stress} verify the claimed properties in Proposition \ref{prop:main}.

Throughout this section, we denote by $C_u$ a large constant depending on the previous solution $( u  ,  R )$. The value of $C_u$ addresses a different norm of $( u  ,  R )$  from line to line, but most importantly, it does not depend on any of the parameters $\sigma, \nu, \mu, \kappa, \omega$ introduced in the velocity perturbation.

\subsection{Concentration and oscillation: choice of parameters}\label{subsec:parameters}

First of all, we assume that $p>1$ as the proposition is the strongest when $p$ is close to $2$. We now specify all the parameters appearing in the velocity perturbation and the constant $r>1$ entering Proposition \ref{prop:main}.

For a small parameter $0<\gamma<1/14$ depending only on $p<2$ as in the Lemma~\ref{lemma:parameters} below, we choose 
\begin{enumerate}

    \item Oscillation $\sigma \in \NN $:
    \[
    \sigma = \lceil \l^{ 2\gamma } \rceil.
    \]
    Without loss of generality, we only consider values of $\l$ such that $\sigma = \l^{ 2\gamma }$ in what follows.
    
    \item Concentration $\kappa,\nu ,\mu$:
    \begin{align*}
    \kappa &= \l^{16 \gamma }\\
    \nu &= \l^{1- 8 \gamma  }\\
    \mu &= \l.
    \end{align*}
    
    \item Acceleration $\omega$:
    \[
     \omega = \lambda.
    \]
\end{enumerate}

For $r>1$ to be fixed in the following lemma, define $q= q(r)>2$ as
\begin{equation} \label{eq:def_q}
\frac{1}{r} = \frac{1}{2} + \frac{1}{q}.
\end{equation}
Now we show that there is an admissible choice for $r>1$ when $\gamma\ll 1$.

\begin{lemma} \label{lemma:parameters}
For all $\gamma>0$ sufficiently small depending only on $1<p<2$, there exists $ 1 < r \leq p $ such that for any $\l\geq 2$ with $\l^{2\gamma} \in \NN$, there hold
\begin{align}
\kappa^{ \frac{1}{2} }   (\nu \mu )^{  \frac{1}{2}  - \frac{1}{p}} &\leq \lambda^{- \gamma  } \qquad (w^{(p)} \in L^\infty L^p) \label{eq:condition_w_p_c_p}\\
\omega^{-1} \sigma   \kappa^{\frac{1}{2}} (\nu \mu)^{1-\frac{1}{p} }&\leq \lambda^{- \gamma  }  \qquad (w^{(a)} \in L^\infty L^{p})\label{eq:condition_w_a_c_p}\\
\omega^{-1} \sigma(\nu \mu)^{\frac{3}{2} - \frac{1}{r} }&\leq \lambda^{- \gamma  }  \qquad (w^{(a)} \in L^2 L^{q})\label{eq:condition_w_a_l_q}\\
\nu\mu^{-1}(\nu \mu)^{1-\frac{1}{r}} &\leq \lambda^{-  \gamma  } \qquad (w^{(c)} \in L^2 L^{q}) \label{eq:condition_w_c_l_q} \\
\sigma  \kappa^{ - \frac{ 1}{2}} \mu (  \nu  \mu )^{\frac{1}{2}  - \frac{1}{r}} &\leq \lambda^{- \gamma  } \qquad \text{Laplacian error for $w^{(p)}$}\label{eq:lap_error}\\
\omega^{-1} \sigma^2   \kappa^{-\frac{1}{2}}  \mu (\nu \mu)^{1 - \frac{1}{r}} &\leq \lambda^{- \gamma  } \qquad \text{Laplacian error for $w^{(a)}$}\label{eq:lap_error_w^a}\\
( \kappa^{\frac{1}{2}} + \omega\sigma^{-1}    \nu) \mu^{-1}(   \nu \mu)^{\frac{1}{2}  - \frac{1}{r}}&\leq \lambda^{-\gamma  } \qquad \text{Acceleration error}\label{eq:tem_error}\\
\sigma^{-1}  (  \nu   \mu  )^{1 - \frac{1}{r}}&\leq \lambda^{- \gamma  } \qquad \text{Oscillation error for $w^{(p)}$} \label{eq:osc_error}\\
\omega^{-1} \sigma^2 \kappa^{\frac{1}{2}} (\nu\mu)^{1-\frac{1}{r}} &\leq  \lambda^{-\gamma} \qquad \text{Oscillation error for $w^{(a)}$} \label{eq:osc_erro_w^a}
\end{align}
\end{lemma}
\begin{proof}
First, notice that \eqref{eq:condition_w_p_c_p} and \eqref{eq:condition_w_a_c_p} hold automatically for any  $\gamma>0$ small enough. Once $\gamma$ is chosen according to \eqref{eq:condition_w_p_c_p} and \eqref{eq:condition_w_a_c_p}, we can find $1 < r \leq p$, close enough to $1$, so that the rest of the conditions hold as well.

Indeed, expressing the left hand sides of the conditions as powers of $\lambda$, we notice that all the exponents are continuous in $r$. So the desired  $r>1$ exists because with $r=1$, the left hand side of  each condition \eqref{eq:condition_w_a_l_q}--\eqref{eq:osc_erro_w^a}  is less than or equal to $\l^{-2\gamma}$.

%Notice that all the left hand side of the conditions are decreasing in $r$ ($q$ is a function of $r$) and  \eqref{eq:condition_w_p_c_p} hold automatically as $\ep \to 0$. 
%Once $\gamma\ll1$ is chosen according to \eqref{eq:condition_w_p_c_p}, 

%\begin{enumerate}
%
 %   \item At $r=1$, \eqref{eq:condition_w_a_c_p} holds with $\l^{-2\gamma}$;
 %   
 %   \item At $r=1$, \eqref{eq:condition_w_c_l_2} holds with $\l^{-24\gamma}$
 %   
 %   \item At $r=1$, \eqref{eq:lap_error} holds with $\l^{-2\gamma}$
 %   
 %   \item At $r=1$, \eqref{eq:tem_error} holds with $\l^{- 2\gamma}$
 %   
 %   \item At $r=1$, \eqref{eq:osc_error} holds with $\l^{- 2\gamma}$
%\end{enumerate}

\end{proof}

\subsection{Estimates on velocity perturbation}

\begin{lemma}\label{lemma:estimates_a_k}
The coefficients $a_k$ are smooth on $ \TT^2  \times[0,1] $ and  
$$
\|  a_k  \|_{C^n(  \TT^2 \times [0,1]   )} \leq C_{u ,n}  \quad \text{for any $n \in \NN $, }
$$
where $C_{u ,n} $ are constants independent of $\l$.
In addition, the bound
$$
\|a_k (t) \|_{L^2(\TT^2)} \lesssim \theta(t)     \Big( \int_{\TT^2 }\rho(x  ,  t   ) \, dx \Big)^{ \frac{1}{2}}
$$
holds uniformly for all time $t\in[0,1]$.
\end{lemma}
\begin{proof}

The smoothness of  $a_k$ follows from definition \eqref{eq:def_a_k}. Since the implicit constant is allowed to depend on $( u  ,  R ) $, the first bound follows immediately from the definition of $a_k$.

The second bound follows from the definition of $a_k$ as well:
$$
\|a_k (t)  \|_{L^2(\TT^2)} \lesssim \theta  \left( \int_{\TT^2} \rho \Gamma_k^2 \left(\Id - \frac{  R }{\rho }\right)  \,dx   \right)^\frac{1}{2}  \lesssim  \theta  \left( \int_{\TT^2} \rho    \,dx   \right)^\frac{1}{2} .
$$

\end{proof}

\begin{proposition} \label{p:estimates_on_w^p}
The principle part $w^{(p)}$ satisfies
\begin{align*}
\| w^{(p)} \|_{L^2( \TT^2  \times [0,1]) }  \lesssim \|  R   \|_{L^1( \TT^2   \times I) }^\frac{1}{2} + C_u \sigma^{-\frac{1}{2}},
\end{align*}
and 
\begin{align*}
\| w^{(p)} \|_{L^\infty ( 0,1; L^p(\TT^2)) }   \lesssim C_u \l^{-\gamma}.
\end{align*}

In particular, for sufficiently large $\l$, 
\begin{align*}
\| w^{(p)} \|_{L^2(  \TT^2 \times [0,1]) } & \lesssim \|  R   \|^{\frac{1}{2}}_{L^1(\TT^2  \times I ) }, \\
\| w^{(p)} \|_{L^\infty( 0,1; L^p(\TT^2)) }  & \leq \frac{\delta}{4}.
\end{align*}
\end{proposition}

\begin{proof}
We first show the estimate for $L^2_{t,x}$  and then for $ L^\infty L^p$.

\noindent
{\bf $L^2_{t,x}$ estimate:}

For each $t\in [0,1]$, we take $L^2$ norm in space and use Lemma~\ref{lemma:improved_Holder} to obtain
\begin{align*}
\| w^{(p)} (t) \|_{L^2( \TT^2) } \lesssim \sum_{k  } g_k  \big( \| a_k (t) \|_{L^2} \| \bwk(t) \|_{L^2} + \sigma^{-\frac{1}{2} }   C_u  \big),
\end{align*}
where we used the fact that $\bwk$ is $\sigma^{-1}\TT^2$-periodic in space.
Recall that $\| \bwk \|_{L^\infty_tL^2} \lesssim 1$. Then using Lemma \ref{lemma:estimates_a_k} and taking $L^2$ norm in time gives
\begin{align}\label{eq:estimate_w_p_1}
\| w^{(p)}  \|_{L^2( \TT^2  \times [0,1]) } \lesssim \sum_{k  } \Big(\int_{0}^1  g_{k}^2  \theta^2(t)\int_{\TT^2} \rho (x  ,  t   ) \,dx \,dt\Big)^\frac{1}{2} 
+ C_u\sigma^{-\frac{1}{2} } .
\end{align}
Notice that
$$
t \mapsto \theta^2(t)\int_{\TT^2} \rho (x  ,  t   ) \,dx 
$$
is a smooth map on $[0,1]$ and $g_k(t)$ is $\sigma^{-1}$-periodic. Thus, we may apply Lemma \ref{lemma:improved_Holder} once again (with $p=1$) to obtain that
\begin{equation}\label{eq:estimate_w_p_2}
\int_{0}^1  g_{k}^2 \theta^2(t)\int_{\TT^2} \rho (x  ,  t   ) \,dx \,dt \lesssim \|  R  \|_{L^1(  \TT^2  \times  I) } +
  C_u   \sigma^{-1}   ,
\end{equation}
where we used the fact that $\int g_k^2 =1$, $\Supp \theta \subset I$, and the bound 
$$
\int_{\TT^2} \rho (x  ,  t   ) \,dx \lesssim \| R  (t)\|_{L^1(\TT^2)} + \| R \|_{L^1(\TT^2 \times I)}.
$$

Thus, combining \eqref{eq:estimate_w_p_1} and \eqref{eq:estimate_w_p_2} gives
\begin{align*}
\| w^{(p)}  \|_{L^2(  \TT^2   \times  [0,1]) } 
&\lesssim  \|  R  \|_{L^1(  \TT^2  \times I) }^\frac{1}{2}  + C_u  \sigma^{-\frac{1}{2} } . 
\end{align*}

\noindent
{\bf $ L^\infty_t  L^p  $ estimate:}

For each $t\in [0,1]$, we take $L^p$ norm in space and then use H\"older's inequality to obtain
\begin{align*}
\| w^{(p)} (t) \|_{L^p( \TT^2) } \lesssim \sum_{k  } \| a_k(t) \|_{L^\infty} |g_k(t)|\| \bwk(t) \|_{L^p}.
\end{align*}
Now, thanks to Lemma~\ref{lemma:estimates_a_k}, \eqref{eq:L^p_bpund_g_k}, and Theorem~\ref{thm:main_thm_for_W_k}, taking $L^\infty$ in time, we get
\begin{align*}
\| w^{(p)}   \|_{L^\infty(0,1;L^p( \TT^2)) } & \lesssim \sum_{k  } \| a_k  \|_{L^\infty_{x,t}} \|g_k\|_{L^\infty}    (\nu \mu )^{  \frac{1}{2}  - \frac{1}{p}} \\
&\leq C_u \kappa^{ \frac{1}{2} }   (\nu \mu )^{  \frac{1}{2}  - \frac{1}{p}}\\
&\lesssim C_u \lambda^{-\gamma},
\end{align*}
where we used \eqref{eq:condition_w_p_c_p} for the last inequality.
\end{proof}

The last estimate of the velocity perturbation concerns the temporal corrector $w^{(t)}$.

\begin{proposition} \label{Prop_Bound_w^t}
The temporal corrector $w^{(t)}=w^{(o)} + w^{(a)}$ satisfies
\begin{align*}
\| w^{(t)} \|_{L^2(0,1;  L^q( \TT^2   )) }  &\lesssim   C_u \l^{- \gamma },\\
\| w^{(t)} \|_{L^\infty(0,1;  L^p( \TT^2   )) }  &\lesssim   C_u \l^{- \gamma },\\
\| w^{(o)} \|_{L^\infty(0,1;  W^{\frac{1}{\delta},\infty}( \TT^2   )) }  &\lesssim   C_u \l^{- \gamma },
\end{align*}
where $q>2$ is as defined in \eqref{eq:def_q}.

In particular, for all sufficiently large $\l$, 
\begin{align*}
\| w^{(t)} \|_{L^2( \TT^2 \times  [0,1]) } & \leq\|  R   \|_{L^1( \TT^2 \times I) }, \\
\| w^{(t)} \|_{L^\infty( 0,1; L^p(\TT^2)) }  & \leq \frac{\delta}{4},\\
\| w^{(o)} \|_{L^\infty(0,1;  W^{\frac{1}{\delta},\infty}( \TT^2   )) }  &\leq \delta.
\end{align*}
\end{proposition}
\begin{proof}

Recall that the temporal corrector $w^{(t)}$ consists of the temporal oscillation corrector $w^{(o)}$ \eqref{eq:def_w_o} and the acceleration corrector $w^{(a)}$ \eqref{eq:def_w_a}:
\begin{align*} 
w^{(o)}  &  =- \sigma^{-1} \theta^2   \mathcal{P}\sum_{k\in \L }    h_k    \D R_k,     \\
w^{(a)}  &  = -\omega^{-1} \sigma  \mathcal{P}\sum_{k \in \L } a_k^2 g_k  |\bwk |^2 \ek.
\end{align*}

The estimate of $w^{(o)}$ follows from a cheap H\"older estimate:
$$
\| w^{(o)} \|_{L^\infty(0,1;  W^{\frac{1}{\sigma},\infty}( \TT^2   ))  }  \lesssim    \sigma^{-1} \sum_{k\in \L }  \| h_k\|_{L^\infty([0,1])} \| \D R_k     \|_{ L^\infty(I;  W^{\frac{1}{\delta},\infty}( \TT^2   ))  } \lesssim C_u  \l^{- \gamma } .
$$

Now we turn to the  $L^2_t L^q $ and $L^\infty_t L^p $ estimates for the acceleration corrector $ w^{(a)}$. For each $t \in [0,1]$  we have
\begin{align*}
 \| w^{(a)}  (t) \|_{L^p(\TT^2)}  \lesssim      \omega^{-1} \sigma    \sum_{k \in \L } \Big\|a_k^2 g_k  |\bwk |^2 \ek \Big  \|_{L^p(\TT^2)},
\end{align*}
since $\mathcal{P}$ is bounded on $L^p(\TT^2)$. Using H\"older's inequality we get
\begin{align*}
 \| w^{(a)}  (t) \|_{L^p(\TT^2)}  &\lesssim      \omega^{-1} \sigma   \sum_{k  } |g_k(t)| \left\|a_k \right\|^2_{L^\infty_{x,t}}  \big\|  |\bwk |^2 \ek \big  \|_{L^\infty(0,1;L^p(\TT^2))} \\
& \lesssim   C_u   \omega^{-1} \sigma   \sum_{k  } |g_k(t)|   \big\|  \bwk  \big  \|^2_{L^\infty(0,1;L^{2p}(\TT^2))}\\
& \lesssim   C_u   \omega^{-1} \sigma  (\nu \mu)^{1 - \frac{1}{p} } \sum_{k  } |g_k(t)|,
\end{align*}
where we used Theorem~\ref{thm:main_thm_for_W_k} at the last step. Now by \eqref{eq:L^p_bpund_g_k} we have
\[
\|w^{(a)}\|_{L^2(0,1; L^q(\TT^2))  } \lesssim C_u   \omega^{-1} \sigma(\nu \mu)^{1 - \frac{1}{q} } \lesssim C_u\l^{-\gamma},
\]
due to \eqref{eq:condition_w_a_l_q}   in Lemma \ref{lemma:parameters}, and
\[
\|w^{(a)}\|_{L^\infty(0,1; L^p(\TT^2))  } \lesssim C_u   \omega^{-1} \sigma  (\nu \mu)^{1-\frac{1}{p} } \kappa^{\frac{1}{2}} \lesssim C_u\l^{-\gamma},
\]
due to \eqref{eq:condition_w_a_c_p}, which are the desired estimates.

\end{proof}

\begin{proposition} \label{Prop_Bound_w^c}
Then divergence-free corrector $w^{(c)}$ satisfies
\begin{align*}
\| w^{(c)} \|_{L^2(0,1; L^{q}  (\TT^2) ) }  \lesssim   C_u   \l^{-\gamma},
\end{align*}
and 
\begin{align*}
\| w^{(c)} \|_{L^\infty( 0,1; L^p(\TT^2)) }  \lesssim C_u   \l^{-\gamma} .
\end{align*}

In particular, for sufficiently large $\l$, 
\begin{align*}
\| w^{(c)} \|_{L^2( \TT^2 \times  [0,1]) } & \leq \|  R   \|_{L^1( \TT^2 \times I) }, \\
\| w^{(c)} \|_{L^\infty( 0,1; L^p(\TT^2)) }  & \leq \frac{\delta}{4}.
\end{align*}
\end{proposition}
\begin{proof}
Using Theorem~\ref{thm:main_thm_for_W_k}, for any $1\leq p \leq \infty $ and $t \in [0,1]$, there holds
\[
\begin{split}
\|w^{(c)} (t)\|_{L^p(\TT^2)} &\lesssim \sum_{k  } \|a_k (t)\|_{L^\infty} |g_k(t)|\|\bwk^{(c)}\|_{L^\infty_tL^p} + \sigma^{-1}\|\nabla^\perp a_k (t) \|_{L^\infty} |g_k(t)|\|\bpk\|_{L^\infty_tL^p}\\
&\lesssim \left[ \nu\mu^{-1}(\nu \mu)^{\frac{1}{2}-\frac{1}{p}} + \mu^{-1}(\nu \mu)^{\frac{1}{2}-\frac{1}{p}} \right] \sum_{k  } \|a_k\|_{C^1_{x,t}}|g_k(t)|.
\end{split}
\]
Recall that $\frac{1}{ r } = \frac{1}{2} + \frac{1}{q}$. Then, using Lemma~\ref{lemma:estimates_a_k} and \eqref{eq:L^p_bpund_g_k}, we have
\[
\|w^{(c)}\|_{L^2( 0,1; L^{q}(\TT^2)) }  \lesssim C_u \nu\mu^{-1}(\nu \mu)^{1-\frac{1}{r}} \lesssim C_u \lambda^{-\gamma},
\]
due to \eqref{eq:condition_w_c_l_q}, and for the specific $ L^\infty_t L^p  $ norm,
\[
\|w^{(c)}\|_{L^\infty( 0,1; L^p (\TT^2)) } \lesssim C_u \kappa^{\frac{1}{2}} \nu\mu^{-1}(\nu \mu)^{\frac{1}{2}-\frac{1}{p}} \lesssim C_u \lambda^{-\gamma},
\]
due to \eqref{eq:condition_w_p_c_p}.

\end{proof}

 \subsection{Sobolev estimates and the exceptional set \texorpdfstring{$E$}{E}}

Finally we show the Sobolev estimates in $ L^\infty( 0,1; W^{s_p,1}(\TT^2)) $ and specify the exceptional set $E$ with the estimates \eqref{eq:Main_Iteration_Estimates_E_density}--\eqref{eq:Main_Iteration_Estimates_E^c_2} in the main proposition.

The proof of the Sobolev estimate follows closely the estimation we have done so far, therefore we only sketch the outline.

\begin{corollary} \label{cor:L^inftyW^{s_p,1}}
Let $s_p: = 1 -40 \gamma$. Then for all sufficiently large $\l$, there holds
\begin{equation*}
\|w \|_{L^\infty( 0,1; W^{s_p,1}(\TT^2))} \leq \delta.
\end{equation*}

\end{corollary}
\begin{proof}
From Lemma \ref{lemma:parameters} and the previous estimations, it is clear that the estimate of $ w$ is majorized by $w^{(p)}$. By a Sobolev interpolation, for any $t\in [0,1]$,
$$
\|w^{(p)} (t)\|_{  W^{s_p,1}(\TT^2)} \lesssim \|w^{(p)} (t)\|_{  L^{ 1} (\TT^2)}^{1-s_p} \|w^{(p)} (t)\|_{  W^{1,1}(\TT^2) }^{s_p}\leq C_u (\sigma \mu)^{s_p} \kappa^{1/2} (\nu \mu)^{-1/2}.
$$

The conclusion would follow if the exponent is negative. Indeed, using the choice of parameters in Section \ref{subsec:parameters}, we see that
$$
(\sigma \mu)^{s_p} \kappa^{1/2} (\nu \mu)^{-1/2} \leq \lambda^{s_p(2\gamma +1) -1 + 20\gamma}, 
$$
which implies
$$
\|w^{(p)} (t)\|_{  W^{s_p,1}(\TT^2)} \leq \delta, \qquad t \in [0,1],
$$
for all sufficiently large $\l$.
\end{proof}

For the exceptional set $E$, recall from the definition of $g_k$, \eqref{eq:def_of_tilde_g_k} and \eqref{eq:def_g_k}, that
\begin{align*}
\Supp g_k \subset  \bigcup_{n \in \ZZ, k\in \L} (t_k, t_k +(\k \sigma)^{-1}  ) + n\sigma^{-1} .
\end{align*}
Note that these open intervals are disjoint by definition.
So we just define $E$ as
\begin{align}\label{eq:def_E}
E:= [0,1] \cap \bigcup_{n \in \ZZ, k\in \L} (t_k, t_k +(\k \sigma)^{-1}  ) + n\sigma^{-1}.
 \end{align}

Then for any $t>0$, we have
\begin{align*}
\mathcal{L}([0,t]\cap E) \leq \sum_{n,k}   (\k \sigma)^{-1},
\end{align*}
where $n,k$ are such that $t_k+ n\sigma^{-1} \in [0,t]$. Since $ \L \ni k$ is finite and fixed, the total number of pairs of $n,k$ is bounded by $C t \sigma$, where $C$ is independent of $t$ and $\kappa$. It follows that
$$
\mathcal{L}([0,t]\cap E) \leq C t \kappa^{-1}
$$
which implies \eqref{eq:Main_Iteration_Estimates_E_density} provided $\l$ is sufficiently large.

Next, we show  \eqref{eq:Main_Iteration_Estimates_E^c} and \eqref{eq:Main_Iteration_Estimates_E^c_2}, namely the Sobolev and energy level estimate away from $E$ .
\begin{lemma}\label{lemma:energy_bounds}
For any $t\in [0,1]\setminus E$
 \begin{align*} 
 \| w (t)\|_{W^{\frac{1}{\delta}, \infty}(\TT^2)} & \leq \delta,  \\
  \Big| \| u_1 (t) \|_{L^2(\TT^2  )}^2 - \| u  (t) \|_{L^2(\TT^2  )}^2 \Big|  & \leq \delta^2,
 \end{align*}
 provided that $\l$ is sufficiently large.
 \end{lemma}
 \begin{proof}
By the definition of $E$, we see that
$$
w^{(p)}= w^{(c)} =w^{(a)} = 0 \qquad \text{for all $t\in [0,1]\setminus E$}.
$$
So, for all $t\in [0,1]\setminus E$ we have 
\begin{align*}
\Big| \| u_1 (t) \|_{L^2(\TT^2  )}^2 - \| u  (t) \|_{L^2(\TT^2  )}^2 \Big| & \leq 
2\Big| \langle w^{(o)}, u  \rangle   \Big| +  \| w^{(o)}(t) \|_{L^2(\TT^2  )}^2, \\
\| w (t)\|_{W^{\frac{1}{\delta}, \infty}(\TT^2)} & \leq \| w^{(o)}(t) \|_{W^{\frac{1}{\delta}, \infty}(\TT^2)},
\end{align*}
and the conclusion follows from Proposition \ref{Prop_Bound_w^t}.
 \end{proof}

\subsection{Estimates on the new Reynolds stress}
The last step of the proof is to estimate $R_1$. We will proceed with the decomposition in Lemma \ref{lemma:new_stress_R_1} and show that for any sufficiently large $\l$, each part of the stress $R_1$ is less than $\frac{\delta}{4} $.

It is also worth noting that we estimate all the errors on $\TT^2 \times [0,1]$ except the oscillation error, for which we only look at $\TT^2 \times I$.

\subsubsection{Linear error}

\begin{lemma}\label{lemma:Linear_error}
For sufficiently large $\lambda$,
\[
\| R_{\Lin} \|_{L^1(0,1; L^r(\TT^2))} \leq \frac{\delta}{4}.
\]
\end{lemma} 

\begin{proof}

We split the linear error into three parts:
\begin{equation*} 
  R_{\Lin}  = \underbrace{ -
\mathcal{R}\left(\Delta w \right)    }_{:=R_{\Lap  }  } +     \underbrace{    \mathcal{R}\big( \p_t (w^{(p)}+ w^{(c)})  \big)   }_{:= R_{\Acc} } +     \underbrace{     \mathcal{R}\left( \D ( w \otimes  u  )+ \D ( u  \otimes w ) \right)    }_{:=R_{\Dri} }.
\end{equation*}
The estimate of the Laplacian error $ R_{\Lap  }$ relies crucially on the temporal concentration $\kappa$, the $R_{\Acc }$ error uses the stream functions $\bpk$ so that we gain a factor of $(\sigma \mu)^{-1}$, and the drifts $R_{\Dri }$ can be handle by standard estimates.

\noindent
{\bf Estimate of  $ R_{\Lap }  $:}

We recall that $w= w^{(p)}   +w^{(c)} +w^{(t)}$ and, using \eqref{eq:appendix_R_2}, decompose the Laplacian error into two parts:
\begin{align*}
\|  R_{\Lap }   \|_{L^1(0,1; L^r(\TT^2))} & \leq 2 \|\nabla w \|_{L^1(0,1; L^r(\TT^2))}  \\
&   \leq 2\|\nabla ( w^{(p)} + w^{(c)} )\|_{L^1(0,1; L^r(\TT^2))}  +2\|\nabla w^{(t)} \|_{L^1(0,1; L^r(\TT^2))}.
\end{align*}
Now recall from \eqref{def_of_w^p+w^c} that
\begin{equation} \label{eq:recall_of_w^p+w^c}
w^{(p)}   +w^{(c)}   = \sigma^{-1}  \sum_{k  } \nabla^\perp \big[a_k g_k  \bpk \big],
\end{equation}
and hence we can  use estimates in Lemma \ref{lemma:estimates_a_k}, \eqref{eq:L^p_bpund_g_k}, and Theorem \ref{thm:main_thm_for_W_k} to obtain

\begin{equation}
\begin{split}
\|\nabla ( w^{(p)} + w^{(c)} )\|_{L^1(0,1; L^r(\TT^2))} &= \sigma^{-1}  \Big\|\sum_{k  } \nabla \nabla^\perp \big[a_k g_k  \bpk \big]\Big\|_{L^1(0,1; L^r(\TT^2))}\\
&\lesssim \sigma^{-1} \sum_{k  } \|    a_k  \|_{C^2_{x,t}} \|g_k\|_{L^1}  \big\|    \bpk   \big\|_{L^\infty(0,1; W^{2,r}(\TT^2))}\\
&\lesssim C_u  \sigma^{-1} \cdot \kappa^{- \frac{1}{ 2}} \cdot (\sigma \mu)^2  \mu^{-1} (\nu \mu)^{ \frac{1}{2} -\frac{1}{r}}.
\end{split}
\end{equation}

As for $ \nabla w^{(t)}$, the oscillation part \eqref{eq:def_w_o} simply satisfies 
\[
\| \nabla w^{(o)} \|_{L^\infty(0,1;L^r(\TT^2) \times [0,1]) }  \lesssim    \sigma^{-1} \sum_{k\in \L }  \| h_k\|_{L^\infty([0,1])} \| \nabla \D R_k     \|_{L^\infty(0,1;L^r( \TT^2)) } \lesssim C_u  \l^{- \gamma },
\]
due to the choice of parameter $\sigma$, and the acceleration part \eqref{eq:def_w_a} enjoys the bound
\[
\begin{split}
 \| \nabla w^{(a)} \|_{L^1(0,1; L^r(\TT^2))} & \lesssim \omega^{-1} \sigma   \sum_{k \in \L } \big\| a_k ^2 \big\|_{C^1_{x,t}}   \|g_k\|_{L^1} \big\| \nabla |\bwk |^2 \ek  \big\|_{ L^\infty(0,1;L^r(\TT^2))}\\
 &\lesssim  C_u  \omega^{-1} \sigma \cdot  \kappa^{-\frac{1}{2}} \cdot (\sigma \mu) (\nu \mu)^{1 - \frac{1}{r}}\\
 &\lesssim C_u \lambda^{-\gamma},
\end{split}
\]
where we used Lemma~\ref{lemma:estimates_a_k}, \eqref{eq:L^p_bpund_g_k}, Theorem~\ref{thm:main_thm_for_W_k}, and Condition \eqref{eq:lap_error_w^a} for the last inequality.

\noindent
{\bf Estimate of  $ R_{\Acc  }  $:}

For the acceleration part of the linear error,
taking time derivative of \eqref{eq:recall_of_w^p+w^c} and using identity \eqref{eq:p_t_of_Psi}, we obtain
\[
\begin{split}
\p_t( w^{(p)}  +w^{(c)})   &= \sigma^{-1} \sum_{k  } \nabla^\perp \big[\p_t (a_k g_k)\bpk    \big] +
\sigma^{-1} \sum_{k  } \nabla^\perp  \big[a_k g_k \p_t\bpk    \big]\\
&=\sigma^{-1} \sum_{k  } \nabla^\perp \big[\p_t (a_k g_k)   \bpk    \big] +
\sigma^{-2} \omega\sum_{k } \nabla^\perp  \big[a_k   g_k^2(\ek \cdot \nabla  )\bpk    \big].
\end{split}
\]
Now thanks to the fact that $ \mathcal{R}\nabla^\perp$ is a Calder\'on-Zygmund operator on $\TT^2$ (see \eqref{eq:C-Z_bound}), we can use Lemma \ref{lemma:estimates_a_k}, \eqref{eq:L^p_bpund_g_k}, and  Theorem \ref{thm:main_thm_for_W_k} to estimate the first term:
\[
\begin{split}
\sigma^{-1} \Big\| \sum_{k  } \mathcal{R}\nabla^\perp \big[\p_t(a_k g_k) \bpk   \big]\Big\|_{  L^1(0,1;L^r) } 
&\lesssim \sigma^{-1}\sum_{k  } \big\|a_k\big\|_{C^1_{x,t}}  \big\|g_k\big\|_{W^{1,1}} \big\|\bpk\big\|_{L^\infty_tL^r }\\
&\lesssim C_u\sigma^{-1} (\sigma \kappa) \kappa^{-\frac{1}{2}} \cdot \mu^{-1}(\nu \mu)^{\frac{1}{2}  - \frac{1}{r}}.
\end{split}
\]
As for the second term, we recall that the  derivative of $\bpk$ in the direction $\ek$ is of order $\sigma \nu$ (rather than $\sigma \mu$ for the full gradient, see \eqref{eq:e_k_derivative_of_Psi}):
\[
\|(\ek \cdot \nabla)\bpk\|_{L^\infty_tL^r} \lesssim \sigma \nu \cdot \mu^{-1}(\nu \mu)^{\frac{1}{2} - \frac1r}.
\]
This estimate together with Lemma \ref{lemma:estimates_a_k} and \eqref{eq:L^p_bpund_g_k} implies
\[
\begin{split}
\sigma^{-2} \omega \Big\| \sum_{k  }   \mathcal{R}\nabla^\perp \big[a_k g_k^2 ( \ek \cdot \nabla  )\bpk  \big]\Big\|_{L^1(0,1;L^r) } 
&\lesssim \sigma^{-2} \omega \sum_{k } \big\|a_k\big\|_{L^\infty_{x,t}} \big\|g_k^2\big\|_{L^1} \big\| ( \ek \cdot \nabla  ) \bpk\big\|_{L^\infty_tL^r }\\
&\lesssim C_u\sigma^{-2}\omega \cdot \sigma \nu \mu^{-1}(\nu \mu)^{\frac{1}{2}  - \frac{1}{r}}.
\end{split}
\]
Due to  \eqref{eq:tem_error}, combining both terms we get
\[
\begin{split}
\| R_{\Acc  } \|_{L^1(0,1; L^r(\TT^2))}&\lesssim C_u\left(   \kappa^{\frac{1}{2}}  \mu^{-1}(\nu \mu)^{\frac{1}{2}  - \frac{1}{r}} + \omega\sigma^{-1}    \nu \mu^{-1}(\nu \mu)^{\frac{1}{2}  - \frac{1}{r}}\right)\\
&\lesssim C_u \lambda^{-\gamma}.
\end{split}
\]

\noindent
{\bf Estimate of  $ R_{\Dri }  $:}

Using the $L^r$ boundedness of  $\mathcal{R}$,
\[
\begin{split}
\|  R_{\Dri }   \|_{L^1(0,1; L^r(\TT^2))} &= \|\mathcal{R}\left( \D ( w \otimes  u  )+ \D ( u  \otimes w ) \right) \|_{L^1(0,1; L^r(\TT^2))}\\
&\lesssim \| w \otimes  u  \|_{L^1(0,1; L^r(\TT^2))}\\
&\lesssim \| w \|_{L^1(0,1; L^r(\TT^2))}    \| u  \|_{L^\infty_{x,t}}\\
&\lesssim C_u \lambda^{-\gamma},
\end{split}
\]
by $L^\infty L^p$ estimates on $w$ in Section~\ref{lemma:estimates_a_k}, and the fact that $r \leq p$.

Combining the above estimates, we can conclude that for all sufficiently large $\l$, there holds
\begin{align*}
\| R_{\Lin} \|_{L^1(0,1; L^r(\TT^2))} \leq  \frac{\delta}{4}.
\end{align*}

\end{proof}

\subsubsection{Correction error}
\begin{lemma}
For sufficiently large $\lambda$,
\[
\| R_{\Cor} \|_{L^1(0,1; L^r(\TT^2))} \leq \frac{\delta}{4}.
\]
\end{lemma} 

\begin{proof}
Using H\"older's inequality, the fact that $\mathcal{R}\D$ is bounded in $L^r$, and \eqref{eq:def_q}, we obtain
\begin{align*}
\| R_{\Cor} \|_{L^1(0,1; L^r(\TT^2))}  &\lesssim  \|(w^{(c)} + w^{(t)})  \otimes w \|_{L^1_t L^r_x }  +     \|  w^{(p)} \otimes (w^{(c)} + w^{(t)} )) \|_{L^1_t L^r_x } \\
& \lesssim  \Big( \|w^{(c)} \|_{L^2_t L^{q }_x} + \|w^{(t)} \|_{L^2_t L^{q}_x } \Big)  \|w \|_{L^2_{x,t}  } \\
&\qquad + \|w^{(p)} \|_{L^2_{x,t}} \Big(  \|w^{(c)} \|_{L^2_t L^{q}_x }  + \|w^{(t)} \|_{L^2_t L^{q}_x }   \Big).
\end{align*}
Recall that $w^{(t)} = w^{(o)} + w^{(a)}$. Then by Propositions~\ref{p:estimates_on_w^p}, \ref{Prop_Bound_w^t}, and \ref{Prop_Bound_w^c} we have
\[
\begin{split}
\|w\|_{L^2_{x,t}} &\leq \|w^{(p)}\|_{L^2_{x,t}} +\|w^{(c)}\|_{L^2_{x,t}}+\|w^{(t)}\|_{L^2_{x,t}}\\
&\lesssim \| R  \|^{\frac{1}{2}}_{L^1(  \TT^2 \times [0,1]) },
\end{split}
\]
as well as
\[
\|w^{(c)}\|_{L^2_t L^{ q }_x }+\|w^{(t)}\|_{L^2_t L^{q }_x } \leq C_u  \lambda^{- \gamma  }  .
\]
Thus for all sufficiently large $\lambda$, we can conclude that
\[
\|R_{\Cor}\|_{L^1(0,1; L^r(\TT^2))} \leq \frac{\delta}{4}.
\]

\end{proof}

\subsubsection{Oscillation error}

As we mentioned before, we estimate the oscillation error on $\TT^2 \times I$ since the perturbations are only designed to balance the old stress on $\TT^2 \times I$, cf. the remainder $R_{\Rem}$ below.
\begin{lemma}
For sufficiently large $\lambda$,
\[
\| R_{\Osc} \|_{L^1(I; L^r(\TT^2))} \leq \frac{\delta}{4}.
\]
\end{lemma}
\begin{proof}
We will use the decomposition from Lemma~\ref{lemma:R_osc_decomposition}
$$
R_{\Osc} = R_{\Osc,x}   + R_{\Osc,t} + R_{\Osc, a}  + R_{\Rem }
$$
where we recall
\begin{align*}
R_{\Osc,x} &=  g_k^2\sum_{k \in \L} \mathcal{B}\Big(\nabla(a_k^2)   ,    \bwk \otimes \bwk  -\fint \bwk \otimes \bwk \Big), \\
R_{\Osc,t} &= - \sigma^{-1}\theta^2  \mathcal{R}\sum_{k\in \L }     h_k      \D \p_t R_k   \\ 
R_{\Osc, a} &= -  \omega^{-1} \sigma   \sum_{k \in \L } \mathcal{R}\left( \p_t \left(a_k^2g_k \right) |\bwk  |^2 \ek \right) \\
R_{\Rem } &=  (1 -\theta^2)      R .
\end{align*}
%%%%%%%%%%%%%%%%%%%%%%%%%%%%%%%%%%%%%%%%%%%%%%
\noindent
{\bf Estimate of  $R_{\Osc,x}$:}

To reduce notations, denote $\mathbf{T}_k : \TT^2 \times [0,1]\to \RR^{2\times 2}$ by 
\[
\mathbf{T}_k =  \bwk   \otimes \bwk   - \fint \bwk \otimes \bwk  , 
\]
so that
\[
R_{\Osc,x}(x  ,  t   ) =  g_k^2\sum_{k \in \L} \mathcal{B}\big(\nabla(a_k^2)  , \mathbf{T}_k \big) .
\]

Using Theorem \ref{thm:bounded_B} and the fact that $\mathbf{T}_k $ has zero spatial mean, we can estimate the $L^r$ norm of $R_{\Osc,x}$ at each time $t \in [0,1]$ as follows.  
\[
\begin{split}
\|R_{\Osc,x} (t )\|_{L^r(\TT^2)} &=  g_k^2\Big\|   \sum_{k } \mathcal{B}\Big(\nabla(a_k^2)  , \mathbf{T}_k \Big)\Big\|_{L^r}\\
&\lesssim g_k^2\sum_{k } \|\nabla (a_k^2)\|_{C^1} \| \mathcal{R}\big(  \mathbf{T}_k  \big)\|_{L^r}\\
&\lesssim \sigma^{-1}g_k^2\sum_{k } \|\nabla (a_k^2)\|_{C^1}  \|\mathbf{T}_k  \|_{L^r},
\end{split}
\]
where the last inequality used the fact that $ \mathbf{T}_k$ has zero spatial mean.
Thanks to Theorem~\ref{thm:main_thm_for_W_k}, for any $k\in \L$
\[
\| \mathbf{T}_k\|_{L^\infty_tL^r} \leq \| \bwk \otimes \bwk   \|_{L^\infty_tL^r} \lesssim \|\bwk\|_{L^
\infty_t L^{2r}}^2 \lesssim  
(\nu \mu)^{1 - \frac{1}{r}}
\]
Therefore, integrating in time and using Lemma~\ref{lemma:estimates_a_k}, \eqref{eq:L^p_bpund_g_k}, and \eqref{eq:osc_error}, we  obtain
\[
\begin{split}
\| R_{\Osc,x} \|_{L^1(0,1; L^r(\TT^2))} &\lesssim C_u \sigma^{-1}  (\nu \mu)^{1 - \frac{1}{r}}\\
&\lesssim C_u \lambda^{-\gamma}.
\end{split}
\]

%%%%%%%%%%%%%%%%%%%%%%%%%%%%%%%%%%%%%%%%%%%%%%%%%%%%%%

%%%%%%%%%%%%%%%%%%%%%%%%%%%%%%%%%%%%%%%%%%%%%%%%%%

\noindent
{\bf Estimate of $R_{\Osc,t}$:} 

Using the bound on $h_k$ \eqref{eq:bound_on_h_k}, we infer
\[
\begin{split}
\|R_{\Osc,t}\|_{L^1(0,1; L^r(\TT^2))} &\leq \Big\|\sigma^{-1} \sum_k h_k  \D \p_t  R_k \Big\|_{L^1(0,1; L^r(\TT^2))}\\
&\lesssim  \sigma^{-1} \Big\|\sum_k h_k\Big\|_{L^1} C_u\\
&\leq  C_u\sigma^{-1}.
\end{split}
\]

%%%%%%%%%%%%%%%%%%%%%%%%%%%%%%%%%%%%%%%%%%%%%
\noindent
{\bf Estimate of $R_{\Osc, a}$:}

Thanks to Theorem~\ref{thm:main_thm_for_W_k} ,
\[
\begin{split}
\|R_{\Osc, a}\| _{L^1(0,1; L^r(\TT^2))}= & \omega^{-1} \sigma    \Big\| \sum_{k  } \mathcal{R}\big( \p_t \left(a_k^2g_k \right)  |\bwk |^2 \ek  \big) \Big\|_{L^1(0,1; L^r(\TT^2))}\\
&\lesssim \omega^{-1} \sigma \sum_{k } \left\|\p_t \left(a_k^2g_k \right)\right\|_{L^1(0,1; C^1(\TT^2))} \|( |\bwk |^2 \ek) \|_{L^\infty_tL^{r}}\\
&= \omega^{-1} \sigma \sum_{k } \left\|\p_t \left(a_k^2g_k \right)\right\|_{L^1(0,1; C^1(\TT^2))} \|  \bwk \|^2_{L^\infty_tL^{2r}}\\
&\lesssim  \omega^{-1} \sigma (\nu\mu)^{1-\frac{1}{r}} \sum_{k } \left\|\p_t \left(a_k^2g_k \right)\right\|_{L^1(0,1; C^1(\TT^2))}.
\end{split}
\]
Using the product rule,
\[
\p_t \left(a_k^2g_k \right) = (\p_t a_k^2) g_k + a_k^2 g_k',
\]
thanks to Lemma ~\ref{lemma:estimates_a_k} and \eqref{eq:L^p_bpund_g_k} we obtain
\begin{equation}
\begin{split}
\left\|\p_t \left(a_k^2g_k \right)\right\|_{L^1(0,1; C^1(\TT^2))} &\lesssim C_u \left(\|g_k\|_{L^1} + \|g_k'\|_{L^1} \right)\\
&\lesssim C_u \big(\kappa^{-\frac{1}{2}} + \sigma \kappa^{\frac{1}{2}} \big).
\end{split}
\end{equation}
Hence
\[
\|R_{\Osc,a }\|_{L^1(0,1; L^r(\TT^2))} \lesssim C_u \omega^{-1}(\nu\mu)^{1-\frac{1}{r}}\cdot \sigma^2 \kappa^{\frac{1}{2}} \lesssim C_u \lambda^{-\gamma},
\]
due to \eqref{eq:osc_erro_w^a}.

%%%%%%%%%%%%%%%%%%%%%%%%%%%%%%%%%%%%%%%%%%%%%
\noindent
{\bf Estimate of $R_{\Rem}$:} 

Finally, the estimate of $R_{\Rem}$ follows from the definition of the cutoff $\theta$ as in \eqref{eq:def_theta}:
\begin{align*}
\|R_{\Rem }\|_{L^1(I; L^r(\TT^2))} &= \|(1 -\theta^2)R\|_{L^1(I; L^r(\TT^2))}\\ 
&\leq \int_{I \cap \{\theta \neq 1\} } \|R(t)\|_{L^\infty(\TT^2)}  \,  dt \\
& \leq  |\{ t \in  I: \theta \neq 1 \} | \|R \|_{L^\infty( \TT^2 \times I)}\\ & \leq \delta/8. 
\end{align*}

Combining this estimate with
\[
\| R_{\Osc,x}\|_{L^1(0,1; L^r(\TT^2))}    +  \| R_{\Osc,t} \|_{L^1(0,1; L^r(\TT^2))}  + \| R_{\Osc, a} \|_{L^1(0,1; L^r(\TT^2))}   \lesssim C_u \lambda^{-\gamma},
\]
we conclude that the desired bound
\[
\| R_{\Osc} \|_{L^1(I; L^r(\TT^2))} \leq \frac{\delta}{4}
\]
holds for all $\lambda$ large enough.
\end{proof}

\appendix

\section{Vanishing viscosity limit in \texorpdfstring{$L^\infty L^2$}{L2}}\label{sec:append_zeroviscosity}
Here, we adopt a proof from \cite{MR3551263} to our settings and show that vanishing viscosity solution of the 2D Euler equations with initial vorticity in $L^p$, $p>1$, conserve the kinetic energy.
 
\begin{theorem} \label{thm:2DEuler_energy_consrvation}
Let $u\in C([0,T];W^{1,1}(\TT^2))$ be a weak solution of the 2D Euler equations with $\curl u(0) \in L^p(\TT^2)$ for some $p>1$, such that there is a family $\{u^\nu\}$ of weak solutions to the 2D Navier-Stokes equations with
\[
\|u^\nu - u\|_{L^\infty(0,T; L^2(\TT^2))} \to 0 \qquad \text{as} \quad \nu \to 0.
\]
Then $u(t)$ conserves the kinetic energy.
\end{theorem}
\begin{proof}
Thanks to the Sobolev embedding $W^{1,1}(\TT^2) \subset L^2(\TT^2) $, we have that $u\in C([0,T];L^2(\TT^2))$, and hence
for any $\varepsilon >0$,  there exists a decomposition $u^\nu = u^\nu_1  +u^\nu_2$ such that
$$
\| u_1^\nu \|_{L^\infty_t L^{2} } \leq \varepsilon \quad \text{ and }\quad  u_2^\nu \in L^\infty_{t,x},
$$
for $\nu$ small enough. Therefore, by the previous theorem, for all $\nu$ small enough, $u^\nu$ is a Leray-Hopf solution of the 2D NSE and is smooth for $t>0$. 

Without loss of generality, assume that $\omega_0 \notin L^2$, as otherwise the result is trivial.  Note that
\[\frac{d}{dt} \|\omega^{\nu} \|_{L^2}^2 = -2\nu\|\nabla\omega^{\nu}\|_{L^2}^2
\lesssim -\nu \| \omega_0\|_{L^p}^{\frac{2p}{p-2}} \|\omega^{\nu}\|_{L^2}^{\frac{4}{2-p}}.\]
Therefore
\[
\| \omega^{\nu}(t)\|_{L^2}^2 \lesssim (\nu t)^{\frac{p-2}{p}}.
\]
Hence
\[
0 \geq \|u^\nu(t)\|_{L^2}^2 - \|u^{\nu}(0)\|_{L^2}^2 \geq  - c(\nu t)^{\frac{2p-2}{p}}.
\]
Taking a limit as $\nu \to 0$, this concludes the proof.
\end{proof}

\section{Some technical inequalities}\label{sec:append_improved_holder}

\subsection{Calderon-Zygmund operators on \texorpdfstring{$\TT^d$}{Td}} For $1<p<\infty$, the classical Calderon-Zygmund estimates hold on $\RR^d$:
\begin{equation*}
    \| \nabla^2  f\|_{L^p(\RR^d)} \lesssim \| \Delta f\|_{L^p(\RR^d)}. 
\end{equation*}

It is standard to transfer the estimates on $\RR^d$ to $\TT^d$:
\begin{equation}\label{eq:Calderon-Zygmund_TT}
    \| \nabla^2  f  \|_{L^p(\TT^d)} \lesssim \| \Delta f\|_{L^p(\TT^d)}. 
\end{equation}

In particular, \eqref{eq:Calderon-Zygmund_TT} implies the bounds used in the paper:
\begin{equation} \label{eq:C-Z_bound}
\| \mathcal{R}\nabla  f \|_{L^p(\TT^2)} \lesssim \|    f \|_{L^p(\TT^2)}\qquad \text{for any $f \in C^\infty(\TT^2)$}, \\
\end{equation}
where $\mathcal{R}$ is the antidivergence operator defined in the next section.

\subsection{Improved H\"older's inequality on \texorpdfstring{$\TT^d$}{Td}}
We recall the following result due to Modena and Sz\'ekelyhidi \cite[Lemma 2.1]{MR3884855}, which extends the first type of such result \cite[Lemma 3.7]{MR3898708}.

\begin{lemma}\label{lemma:improved_Holder}
Let $d \geq 2$, $r \in [1,\infty]$, and $a,f :\TT^d \to \RR$ be smooth functions. Then for every $\sigma \in \NN$,
\begin{equation}
\Big|   \|a f(\sigma \cdot ) \|_{r }  - \|a \|_{r} \| f \|_{r } \Big|\lesssim \sigma^{-\frac{1}{r}} \| a\|_{C^1} \| f \|_{ r }.
\end{equation}
\end{lemma}

This result is used to control the energy of the perturbations in Section \ref{sec:proof_step_2}. Note that the error term on the right-hand side can be made arbitrarily small by increasing the oscillation $ \sigma$.

\section{Antidivergence operators}\label{sec:append_antidiv}

For any $f\in C^\infty(\TT^2)$, there exists a $v \in C^\infty_0(\TT^2)$ such that
$$
\Delta v = f - \fint_{\TT^2} f.
$$
And we denote this solution $v$ by $\Delta^{-1}f$. Note that if $f\in C^\infty_0(\TT^2) $, then by rescaling  we have 
$$
\Delta^{-1} \big( f(\sigma \cdot )  \big) = \sigma^{-2}  v(\sigma \cdot ) \quad \text{ for $\sigma \in \NN$.}
$$

\subsection{Tensor-valued antidivergence \texorpdfstring{$ \mathcal{R}$}{R}}

We recall the following antidivergence operator $\mathcal{R}$ introduced in \cite{MR3090182}.

\begin{definition}
$\mathcal{R} : C^\infty(\TT^2 ,\RR^2) \to C^\infty(\TT^2, \mathcal{S}^{ 2 \times 2 }_0)$ is defined by
\begin{equation}\label{eq:appendix_R_def}
(\mathcal{R} v )_{ij} =  \mathcal{R}_{ijk} v_k
\end{equation}
where
$$
 \mathcal{R}_{ijk} =    -\Delta^{-1} \p_k \delta_{ij} + \Delta^{-1} \p_i \delta_{jk} + \Delta^{-1} \p_j \delta_{ik} .
$$
\end{definition}

It is clear that $\mathcal{R} $ is well-defined since $\mathcal{R}_{ijk}$ is symmetric in $i,j$ and taking the trace gives
\begin{align*}
\Tr    \mathcal{R} v&=    -2 \Delta^{-1} \partial_k   v_k  +   \Delta^{-1}   \partial_k v_k + \Delta^{-1}   \partial_k v_k=0.
\end{align*}
By a direct computation, one can also show that
$$
\D (\mathcal{R} v  ) = v - \fint_{\TT^2} v  \quad \text{for any $v \in C^\infty(\TT^2 ,\RR^2)$}
$$
and
\begin{equation}\label{eq:appendix_R_2}
\mathcal{R} \Delta v   = \nabla v + \nabla v ^T  \quad \text{for any divergence-free $v \in C^\infty(\TT^2 ,\RR^2)$}.
\end{equation}

The antidivergence operator $ \mathcal{R}$ is bounded on  $L^{p} (\TT^2) $ for any $1\leq p \leq \infty$~\cite{MR3884855}.
\begin{theorem}\label{thm:bounded_R}
Let $1 \leq p \leq \infty$. For any vector field $f \in C^\infty_0(\TT^2,\RR^2)$, there holds
$$
\| \mathcal{R} f \|_{L^{p}(\TT^2 )} \lesssim \|  f \|_{L^{p}(\TT^2 )}.
$$

In particular, if $f \in C^\infty_0(\TT^2 ,\RR^2)$, then
$$
\| \mathcal{R} f(\sigma \cdot ) \|_{L^{p}(\TT^2 )} \lesssim \sigma^{-1}\|  f \|_{L^{p}(\TT^2 )} \quad \text{for any $\sigma \in \NN$}.
$$

\end{theorem}

\subsection{Bilinear antidivergence \texorpdfstring{$\mathcal{B}$}{B}}

We can also introduce the bilinear version $\mathcal{B} : C^\infty(\TT^2, \RR^2) \times C^\infty(\TT^2   ,\RR^{2\times 2} ) \to C^\infty(\TT^2, \mathcal{S}^{ 2 \times 2 }_0) $ of $\mathcal{R}$.

Let
\begin{equation*}
( \mathcal{B}( v,A ))_{i j}  = v_l \mathcal{R}_{ijk}A_{lk} - \mathcal{R}( \p_i v_l \mathcal{R}_{ijk}A_{lk} ) 
\end{equation*}
or by a slight abuse of notations
$$
  \mathcal{B}( v,A )  = v  \mathcal{R} A  - \mathcal{R}( \nabla v  \mathcal{R} A   ) .
$$
This bilinear antidivergence $\mathcal{B}$ allows us to gain derivative when the later argument has zero mean and a small period.

\begin{theorem}\label{thm:bounded_B}
Let  $1 \leq p \leq \infty$. For any $v \in C^\infty(\TT^2,\RR^2)$ and $A \in C^\infty_0(\TT^2,\RR^{2\times 2})$, 
\begin{equation} \label{eq:div_B_tensor}
\D( \mathcal{B}(v , A)  )  =  v   A   - \fint_{\TT^2} vA ,
\end{equation}
and
$$
\|  \mathcal{B} (v ,A) \|_{L^{p}(\TT^2 )} \lesssim \|   v \|_{C^{1}(\TT^2 )} \|   \mathcal{R} A \|_{L^{p}(\TT^2 )}.
$$
\end{theorem}
\begin{proof}
A direct computation gives
\begin{align*}
\D( \mathcal{B}(v , A)  ) & = v_l \p_j \mathcal{R}_{ijk}A_{lk}  +   \p_j v_l \mathcal{R}_{ijk}A_{lk}- \D \mathcal{R}( \p_i v_l \mathcal{R}_{ijk}A_{lk} ) \\
&=     v_l  A_{il}    + \fint \p_i v_l \mathcal{R}_{ijk}A_{lk}
\end{align*}
where we have used the fact that $A$ has zero mean and $\mathcal{R}$ is symmetric.

Integrating by parts, we have
\begin{align*}
 \fint \p_i v_l \mathcal{R}_{ijk}A_{lk} =  -\fint  v_l \p_i \mathcal{R}_{ijk}A_{lk} =-\fint  v_l A_{lj},
\end{align*}
which implies that
\begin{align*}
\D( \mathcal{B}(v , A)  )   
 =     v A -       \fint v A.
\end{align*}

The second estimate follows immediately from the definition of $\mathcal{B}$ and Theorem~\ref{thm:bounded_R}.

\end{proof}

\subsection*{Conflicts of interest}The authors have no relevant financial or non-financial interests to disclose.

\bibliographystyle{alpha}
\bibliography{2D_NSE} 
%\nocite{*}

\end{document}